\documentclass[11pt]{article} 

\usepackage{amsmath, comment}
\usepackage{amssymb}
\usepackage{amstext}
\usepackage{xspace}
\usepackage{url}
\usepackage[all,2cell]{xy}\UseAllTwocells
\xyoption{all}
\usepackage{proof}
\usepackage{stmaryrd}
\usepackage{rotating}

\usepackage{setspace}
\usepackage[left=2.5cm,top=3cm,right=3cm,bottom=2.5cm]{geometry}
\newif\ifignore 
\ignorefalse

\newcommand{\auxproof}[1]{
\ifignore\mbox{}\newline
\textbf{PROOF:} \dotfill\newline
{\it #1}\mbox{}\newline
\textbf{ENDPROOF}\dotfill
\fi}

\newtheorem{theorem}{Theorem}
\newtheorem{lemma}[theorem]{Lemma}
\newtheorem{proposition}[theorem]{Proposition}
\newtheorem{corollary}[theorem]{Corollary}
\newtheorem{definition}[theorem]{Definition}

\newenvironment{proof}[1][Proof]%
   { \begin{trivlist}%
     \item[\hskip \labelsep {\bfseries #1}]%
   }%
   { \end{trivlist}%
   }
\newcommand{\QEDbox}{\square}
\newcommand{\QED}{\hspace*{\fill}$\QEDbox$}

\newcommand{\after}{\mathrel{\circ}}
\newcommand{\cat}[1]{\ensuremath{\mathbf{#1}}}
\newcommand{\id}{\ensuremath{\mathrm{id}}}

\newcommand{\tuple}[2]{\ensuremath{\langle #1, #2 \rangle}}

\newcommand{\DL}{\cat{DL}}
\newcommand{\pDL}{\cat{CHyp}}
\newcommand{\Coh}{\cat{Coh}}

\newcommand{\Heyt}{\cat{Heyt}}
\newcommand{\pHA}{\cat{FHyp}}

\newcommand{\Aa}{\cat{L}}
\newcommand{\B}{\cat{B}}
\newcommand{\C}{\cat{C}}
\newcommand{\D}{\cat{D}}
\renewcommand{\L}{L}
\newcommand{\K}{K}
\newcommand{\Set}{\cat{Set}}
\newcommand{\Sets}{\cat{Set}}
\newcommand{\Sp}{\mathcal{K}}

\newcommand{\CE}{(\_)^\delta}

\renewcommand{\S}{\mathcal{S}}
\newcommand{\A}{\mathcal{A}}
\newcommand{\Op}[1]{I({#1}^\delta)}
\newcommand{\Cl}[1]{F({#1}^\delta)}
\newcommand{\F}{F\!m}
\newcommand{\op}{\mathrm{op}}

\newcommand{\Th}{\mathbb{T}}

\newenvironment{bijectivecorrespondence}
  {\newif\ifbijnotfirst
   \global\bijnotfirstfalse
   \global\def\bijprev{}
   \normalsize
   \begin{tabular}{cl}}
  {\end{tabular}
   }
\newcommand{\correspondence}[2][]{%
  \ifbijnotfirst%
    \rule{0pt}{5.8pt}%
    \smash{\ensuremath{\infer={\hphantom{#2}}{\hphantom{\bijprev}}}} \\%
  \fi%
  \global\bijnotfirsttrue%
  \global\def\bijprev{#2}%
  \ensuremath{#2} & #1 \\%
  }

\title{Generalising canonical extension to the categorical setting}

\author{Dion Coumans\thanks{\it Institute for Mathematics, Astrophysics and Particle Physics, Radboud Universiteit Nijmegen, P.O. Box 9010, 6500 GL Nijmegen, The Netherlands, d.coumans@math.ru.nl}}

\date{\small \today}

\begin{document}
\maketitle

\begin{abstract}
Canonical extension has proven to be a powerful tool in algebraic study of propositional logics. In this paper we describe a generalisation of the theory of canonical extension to the setting of first order logic. We define a notion of canonical extension for \emph{coherent categories}. These are the categorical analogues of distributive lattices and they provide categorical semantics for coherent logic, the fragment of first order logic in the connectives $\wedge$, $\vee$, $0$, $1$ and $\exists$. We describe a universal property of our construction and show that it generalises the existing notion of canonical extension for distributive lattices. 
Our new construction for coherent categories has led us to an alternative description of the \emph{topos of types}, introduced by Makkai in \cite{Mak}. This allows us to give new and transparent proofs of some properties of the action of the topos of types construction on morphisms. Furthermore, we prove a new result relating, for a coherent category $\C$, its topos of types to its category of models (in $\Sets$).
\end{abstract}

\setlength{\parindent}{0pt}
\section{Introduction}
In 1951, J\'onsson and Tarski introduced the notion of canonical extension for Boolean algebras with operators \cite{JoTa}. Canonical extension provides an algebraic formulation of duality theory and a tool to derive representation theorems. In the last decades the theory of canonical extensions has been simplified and generalised and it is now applicable in the broad setting of distributive lattices and even partially ordered sets \cite{GeJo94,GePr08, DGP}. Canonical extension has proven to be a powerful tool in the algebraic study of propositional logics. Generalising the notion of canonical extension to the categorical setting opens the way to the application of those techniques in the study of predicate logics. 

We focus on distributive lattices and their categorical counterparts. For a distributive lattice $\L$, its canonical extension $\L^\delta$ may be concretely described as the downset lattice of the poset $(PrFl(L), \supseteq)$ of prime filters of $\L$ ordered by reverse inclusion. The assignment $\L \mapsto \L^\delta$ extends to a functor $\CE \colon \DL \to \DL^+$, from the category of distributive lattices to the category of completely distributive algebraic lattices, which is left adjoint to the forgetful functor $\DL^+ \to \DL$. 

In this paper we define a notion of canonical extension for \emph{coherent categories}, the categorical analogues of distributive lattices. These provide categorical semantics for coherent logic, the fragment of first order logic in the connectives $\wedge$, $\vee$, $0$, $1$ and $\exists$. Our construction is inspired by the work of Pitts. In \cite{Pi83, Pi89}, he defines, for a coherent category $\C$, its \emph{topos of filters} $\Phi(\C)$, which is a categorical generalisation of the functor which sends a lattice $\L$ to the lattice $Idl(Fl(\L))$ of ideals of the lattice of filters of $\L$. In his description of $\Phi(\C)$ he exploits the correspondence between coherent categories and 
so-called coherent hyperdoctrines. This correspondence also forms the basis of our construction, which we describe in Section~\ref{sec:extcoh}. We show that our notion of canonical extension for coherent categories extends the existing notion of canonical extension for distributive lattices, which, viewed as categories, are coherent categories. Furthermore, we prove that our construction may be characterised by a universal property, which is similar to the one known from the algebraic setting.

In Section~\ref{sec:Heyt} we restrict our attention to Heyting categories. These provide sound and complete semantics for intuitionistic first order logic. Heyting categories form a non-full subcategory of the category of coherent categories. For any coherent category, its canonical extension is a Heyting category. Furthermore, we prove that the canonical extension of a morphism of Heyting categories is again a morphism of Heyting categories. 

In  \cite{Mak} Makkai introduces, for a coherent category $\C$, the \emph{topos of types} $T(\C)$ of $\C$. Magnan claims in his thesis \cite{Mag} that the topos of types construction is a natural generalisation of canonical extension to the categorical setting. Furthermore, at a talk at PSSL in 1999, Magnan's PhD advisor Reyes announced (but did not prove) that this construction may be used to prove interpolation for different first order logics  \cite{Rey}. In \cite{Bu04}, Butz gives a logical description of Makkai's topos of types, also drawing attention to the connection with canonical extension. 
The topos of types construction is closely related to our construction of canonical extension of coherent categories and this has led us to an alternative description of the topos of types. We work this out in Section~\ref{sec:tt}. 

Our alternative description of Makkai's construction sheds new light on some of its properties, as we illustrate in Section~\ref{sec:proptt}. We study the action of the topos of types construction on morphisms. Furthermore, we prove a new result relating, for a coherent category $\C$, the topos of types $T(\C)$ of $\C$ and the category of models of $\C$ (in $\Set$). 

Pitts used the topos of filters construction to give a (new) proof of the fact that intuitionistic first order logic has the interpolation property. However, in his proof the topos of filters could be replaced by the topos of types. Based on the algebraic situation, we expect that, compared to the topos of filters, the topos of types construction behaves better with respect to preservation of additional axioms. Therefore we hope to apply this construction in the study of (open) interpolation problems for first order logics. This is left to future work. 

\section{Canonical extension for distributive lattices}
\label{sec:canextdl}
We recall some essential facts about canonical extension for lattices, which we rely on in the remainder. Along the way we introduce some notation. For more details the reader may consult \cite{GeHa, GeJo94, GePr08}. In this paper, all lattices are assumed to be bounded.

\begin{definition}
Let $\L$ be a lattice. A \emph{canonical extension of $\L$} is an embedding $e \colon \L \hookrightarrow C$, of $\L$ in a complete lattice $C$ which satisfies the following two conditions:
\begin{enumerate}
\renewcommand{\theenumi}{\roman{enumi}}
\item every element of $C$ may be written both as a join of meets and as a meet of joins of elements in the image of $e$ ($e$ is \emph{dense});
\item for all $F, I \subseteq \L$, if $\bigwedge e[F] \le \bigvee e[I]$, then there exist finite subsets $F' \subseteq F$ and $I' \subseteq I$ such that $\bigwedge F' \le \bigvee I'$ ($e$ is \emph{compact}).
\end{enumerate}
\end{definition}

\begin{theorem}
Every lattice has an (up-to-isomorphism) unique canonical extension. 
\end{theorem}

We denote the canonical extension of $\L$ by $e_\L \colon \L \hookrightarrow \L^\delta$. From now on we view $\L$ as a subset of its canonical extension and, for $a \in \L$, we just write $a$ for the image $e_\L(a)$ in $\L^\delta$.

We introduce some terminology and notation which will be useful in the remainder. We write $\Cl{L}$ for the meet closure of $\L$ in $\L^\delta$ and $\Op{L}$ for the join closure of $\L$ in $\L^\delta$, {\it i.e.},
$$
\begin{array}{rcl}
\Cl{L} &=& \{u \in \L^\delta \,|\, u = \bigwedge \{a \in \L \,|\, u \le a\}\},\\
\Op{L} &=& \{u \in \L^\delta \,|\, u = \bigvee \{a \in \L \,|\, a \le u\}\}.
\end{array}
$$   
The poset $\Cl{L}$ (with the induced order) is isomorphic to the lattice $Fl(\L)$ of filters of $\L$ ordered by \emph{reverse} inclusion. This isomorphism is given by
\begin{equation}
\label{isoflt}
\begin{array}{rcll}
Fl(\L) & \leftrightarrows & F(\L^\delta)\\
F &\mapsto& \bigwedge F =: x_F&\text{(viewing $F$ as a subset of $\L^\delta$)}\\
F_x := \{a \in \L\,|\,x \le a\} &\mapsfrom& x.
\end{array}
\end{equation} 
The elements of $\Cl{L}$ are called filter elements of $\L^\delta$. Dually, the poset $\Op{L}$ is isomorphic to the lattice $Idl(L)$ of ideals of $\L$ ordered by inclusion. The elements of $\Op{L}$ are called ideal elements of $\L^\delta$.

The denseness of $e_\L$ gives us, in principle, two natural ways to lift an order preserving map $f \colon \L \to \K$ between lattices to an order preserving map between their canonical extensions. These two extensions, which are denoted $f^\sigma$ and $f^\pi$, are defined as follows,
$$
\begin{array}{rcll}
f^\sigma(x) &=& \bigwedge\{f(a)\,|\,x \le a \in \L\} &\text{for $x \in \Cl{L}$},\\
f^\sigma(u) &=& \bigvee\{f^\sigma(x)\,|\,u \ge x \in \Cl{L}\} &\text{for $u \in \L^\delta$},\\
\\
f^\pi(y) &=& \bigvee\{f(a)\,|\,y \ge a \in \L\} &\text{for $y \in \Op{L}$},\\
f^\pi(u) &=& \bigwedge\{f^\pi(y)\,|\,u \le y \in \Op{L}\} &\text{for $u \in \L^\delta$}.\\
\end{array}
$$

The mappings $f \mapsto f^\sigma$ and $f \mapsto f^\pi$ do not preserve composition in general. The following proposition describes situations in which they do. 

\begin{proposition}
\label{prop:prescomp}
Let $K, L$ and $M$ be lattices with order-preserving maps $M \xrightarrow{g} L^n \xrightarrow{f} K$. If $f$ preserves finite joins in each coordinate (while keeping the other coordinates fixed), then \mbox{$(f \after g)^\sigma = f^\sigma \after g^\sigma$}. Dually, if $f$ preserves finite meets in each coordinate, then $(f \after g)^\pi = f^\pi \after g^\pi$. 
\end{proposition}

For various order-preserving maps between lattices, their $\sigma$-extension and $\pi$-extension coincide, thus giving a unique lifting. 

\begin{proposition}
Let $f \colon \L \to \K$ be an order-preserving map between lattices. If $f$ is finite join preserving (resp. finite meet preserving), then $f^\sigma = f^\pi$. In this case we denote this unique extension by $f^\delta$. Furthermore, $f^\delta$ is completely join preserving (resp. completely meet preserving). 
\end{proposition}

In particular, for a lattice homomorphism $f \colon \L \to \K$, the map $f^\delta \colon \L^\delta \to \K^\delta$ is a complete lattice homomorphism. The assignment $\L \mapsto \L^\delta$ extends to a functor $\cat{Lat} \to \cat{CLat}$, from the category of lattices to the category of complete lattices, which preserves finite products. We rely on this property when extending (non-unary) operations on a lattice to its canonical extension.
  
We now restrict our attention to distributive lattices. The category of distributive lattices with lattice homomorphisms is denoted by $\DL$. For a distributive lattice $\L$, its canonical extension $\L^\delta$ may be concretely described as the downset lattice of the poset \mbox{$(PrFl(\L), \supseteq)$} of prime filters of $\L$ ordered by reverse inclusion. In this setting, the isomorphism \eqref{isoflt} restricts to an isomorphism between the prime filters $PrFl(L)$ of $L$ and the completely join-irreducible elements $J^\infty(L^\delta)$ of $\L^\delta$. The lattice $\L^\delta$ is completely distributive and algebraic. We write $\DL^+$ for the category of completely distributive algebraic lattices with complete homomorphisms.

\begin{theorem}
\label{thm:extlat}
The assignment $\L \mapsto \L^\delta$ extends to a functor $\CE \colon \DL \to \DL^+$ which is left adjoint to the inclusion functor $i \colon \DL^+ \to \DL$.  
\end{theorem}

Occasionally, we also write $(\_)^\delta$ for the composition $\DL \xrightarrow{(\_)^\delta} \DL^+ \xrightarrow{i} \DL$.
Theorem~\ref{thm:extlat} implies that, for $\L \in \DL$ and $K \in \DL^+$, every lattice homomorphism $f \colon L \to K$ extends uniquely to a complete lattice homomorphism $\overline{f} \colon L^\delta \to K$. This yields an isomorphism
\begin{equation}
\label{eq:natdl}
\overline{(\_)} \colon Hom_{\DL}(L, K) \to Hom_{\DL^+}(L^\delta, K),
\end{equation}
natural in $L$ and $K$. 

In this paper we consider join-preserving maps between distributive lattices. We end this section by describing how these interact with the natural isomorphism \eqref{eq:natdl}. The following result was already known to Mai Gehrke and John Harding, but never published.

\begin{proposition}
\label{prop:comjpm}
Consider the following commutative diagram
$$
\xymatrix{
L_1 \ar[r]^-{f} \ar[d]_-{h_1} & L_2 \ar[d]^-{h_2}\\
K_1 \ar[r]_-{g} & K_2
}
$$
where $L_1, L_2 \in \DL$, $K_1, K_2 \in \DL^+$, $h_1, h_2$ are lattice homomorphisms, $f$ is a finite join-preserving map and $g$ is a completely join-preserving map. The following are equivalent
\begin{enumerate}
\item for all prime filters $\rho$ in $L_1$,
$$
g(\bigwedge \{h_1(a)\,|\,a\in\rho\}) = \bigwedge\{g(h_1(a))\,|\,a \in \rho\}
$$
\item $g \after \overline{h_1} = \overline{h_2} \after f^\delta$. 
\end{enumerate}
\end{proposition}  

\begin{proof}
As $g$, $\overline{h_1}$, $\overline{h_2}$ and $f^\delta$ are all completely join preserving and $L_1^\delta$ is join-generated by its completely join-irreducible elements, the second condition is equivalent to, for all $x \in J^\infty(L^\delta)$, 
$$g(\overline{h_1}(x)) = \overline{h_2}(f^\delta(x)).$$
Let $x \in J^\infty(L_1^\delta)$ and let $\rho_x = \{a \in L_1 \,|\, x \le a\}$ be the corresponding prime filter in $L_1$. Then
$$
\begin{array}{rcll}
g(\overline{h_1}(x))
&=&
g(\overline{h_1}(\bigwedge\{a\,|\,a \in \rho_x\}))\\
&=&
g(\bigwedge\{h_1(a)\,|\,a \in \rho_x\})&\text{($\overline{h_1}$ is a complete homomorphism extending $h_1$).}\\
\\
\overline{h_2}(f^\delta(x))
&=&
\overline{h_2}(\bigwedge\{f(a)\,|\,a \in \rho_x\})&\text{(definition of $f^\delta = f^\sigma$)}\\
&=&
\bigwedge\{h_2(f(a))\,|\,a \in \rho_x\}\\
&=&
\bigwedge\{g(h_1(a))\,|\,a \in \rho_x\}.
\end{array}
$$
Hence, $g(\overline{h_1}(x)) = \overline{h_2}(f^\delta(x))$ iff $g(\bigwedge \{h_1(a)\,|\,a\in\rho_x\}) = \bigwedge\{g(h_1(a))\,|\,a \in \rho_x\}$. The claim now follows using the one-to-one correspondence between completely join-irreducible elements of $L_1^\delta$ and prime filters of $L_1$.\QED
\end{proof}

\section{Canonical extension for coherent categories}
\label{sec:extcoh}
The categorical analogue of a distributive lattice is a \emph{coherent category}, {\it i.e.}, a category $\C$ which has finite limits, stable images and the property that, for all $A \in \cat{C}$, $Sub_\C(A)$ has stable finite joins. We write  $\Coh$ for the category of all (small) coherent categories, with structure preserving functors. Remark that a distributive lattice (viewed as a category) is a coherent category. As we describe below, there is a correspondence between coherent categories and so-called coherent hyperdoctrines. We exploit this correspondence to define a notion of canonical extension for coherent categories. 

Coherent categories provide semantics for \emph{coherent logic}, the fragment of first order logic with only the connectives $\wedge$, $\vee$, $\top$, $\bot$ and $\exists$.  In a coherent category $\C$, for each $A \in \C$, $Sub_\C(A)$ is a distributive lattice. This enables one to interpret the propositional connectives. As $\C$ has images, for each $f \colon A \to B$, the pullback functor $f^* \colon Sub_\C(B) \to Sub_\C(A)$ has a left adjoint $\exists_f$, which enables the interpretation of the existential quantifier  (see {\it e.g.} D1 in \cite{Joh}). 

For a coherent category $\C$, the functor $Sub_\C \colon \C^{\op} \to \DL$, which sends an object of $\cat{C}$ to the distributive lattice of its subobjects, is a {\it coherent hyperdoctrine} from which we may recover $\C$ (up to equivalence), as is made precise in Proposition~\ref{prop:adjpdlcoh}.

\begin{definition}
\label{def:cohhyp}
Let $\cat{B}$ be a category with finite limits. A \emph{coherent hyperdoctrine} over $\cat{B}$ is a functor $P \colon \cat{B}^{\op} \to \DL$ such that,  for every morphism $A \xrightarrow{\alpha} B$ in $\cat{B}$, $P(\alpha) \colon P(B) \to P(A)$ has a left adjoint $\exists_\alpha^P$ satisfying
\begin{enumerate}
\item Frobenius reciprocity, {\it i.e.}, for all $a \in P(A)$, $b \in P(B)$,
$$
\exists_\alpha^P(a \wedge P(\alpha)(b)) = \exists_\alpha^P(a) \wedge b
$$
\item Beck-Chevalley condition, {\it i.e.}, for every pullback square 
$$
\xymatrix{
Q \ar[r]^-{\alpha'} \ar[d]_-{\beta'} & B \ar[d]^-{\beta}\\
A \ar[r]_-{\alpha} & C
}
$$
in $\cat{B}$, $P(\beta) \after \exists_\alpha^P = \exists_{\alpha'}^P \after P(\beta')$.
\end{enumerate}
We often omit the superscript $P$ in $\exists_\alpha^P$.
A \emph{coherent hyperdoctrine morphism} from $P_1 \colon \cat{B}_1^{\op} \to \cat{DL}$ to $P_2 \colon \cat{B}_2^{\op} \to \cat{DL}$ is a pair $(K, \tau)$, where $K \colon \cat{B}_1 \to \cat{B}_2$ is a finite limit preserving functor and $\tau \colon P_1 \to P_2\after K$ is a natural transformation satisfying, for all $A \xrightarrow{\alpha} B$ in $\B_1$, 
$$
\exists_{K(\alpha)}^{P_2} \after \tau_A = \tau_B \after \exists_\alpha^{P_1}.
$$
We write $\pDL$ for the category of coherent hyperdoctrines.
\end{definition}

A theory $\Th$ in coherent logic naturally gives rise to a coherent hyperdoctrine $\F_\Th \colon \B^{\op} \to \DL$. The objects of $\B$ are pairs $\langle \vec{x}, \vec{s} = \vec{t} \rangle$, where $\vec{x}$ is a finite sequence of variables (context) and $\vec{s}$ and $\vec{t}$ are finite sequences of terms (of the same length) in context $\vec{x}$. A morphism $\tuple{\vec{x}}{\vec{s} = \vec{t}\,} \to \tuple{\vec{y}}{\vec{u} = \vec{v}}$ is an equivalence class $[\vec{w}]$, where $\vec{w}$ is a finite sequence of terms in context $\vec{x}$ that is of the same length as $\vec{y}$ such that the following sequent is derivable in $\Th$
$$\vec{s} = \vec{t} \,\,\vdash_{\vec{x}}\,\, \vec{u}\,[\vec{w}/\vec{y}] = \vec{v}\,[\vec{w}/\vec{y}],
$$
where, for example, $\vec{s} = \vec{t}$ denotes the conjunction $s_0 = t_0 \wedge \ldots \wedge s_{n-1} = t_{n-1}$ (with $n = length(s)$) and $\vec{u}\,[\vec{w}/\vec{y}]$ is obtained by substituting $w_i$ for $y_i$ in $\vec{u}$, for each $i < length(y)$.
Two such sequences $\vec{w}$ and $\vec{z}$ are equivalent iff the following sequent is derivable in $\Th$
$$
\vec{s} = \vec{t} \,\,\vdash_{\vec{x}}\,\, \vec{w} = \vec{z}.
$$
We say a coherent formula $\psi$ is a \emph{formula in context $\vec{x}$} iff all free variables of $\psi$ are among $\vec{x}$. For an object $\langle \vec{x}, \vec{s} = \vec{t} \rangle$, the underlying set of $\F_\Th(\langle \vec{x}, \vec{s} = \vec{t} \rangle)$ is the collection of coherent formulae in context $\vec{x}$, modulo the equivalence relation $\sim$, where, for coherent formulas $\phi, \psi$ in context $\vec{x}$,
$$
\phi \sim \psi \quad\Leftrightarrow\quad \phi \wedge \vec{s} = \vec{t} \,\,\vdash_{\vec{x}}\,\, \psi \,\,\,\text{and}\,\,\,\psi \wedge \vec{s} = \vec{t} \,\,\vdash_{\vec{x}}\,\, \phi \,\,\,\,\text{are derivable in $\Th$}.
$$
The derivation rules of coherent logic ensure that $\F_\Th(\langle \vec{x}, \vec{s} = \vec{t} \rangle)$, ordered by derivability, is a distributive lattice. For a morphism $\alpha$ in $\B$, $\F_\Th(\alpha)$ is given by substitution of terms in formulae and its adjoints may be described using the existential quantification and equality of the logic. The Beck-Chevalley condition corresponds to the fact that quantification and substitution interact appropriately and the validity of the Frobenius axiom is ensured by the derivation rules for coherent logic.

As stated above, for a coherent hyperdoctrine $\C$, the subobject functor $Sub_\C \colon \C^\op \to \DL$ is a coherent hyperdoctrine. In particular, the coherent category $\Set$ gives rise to a coherent hyperdoctrine $Sub_\Set$, which maps a set $A$ to its powerset $\mathcal{P}(A)$ and a function $A \xrightarrow{f} B$ to the inverse image function $\mathcal{P}(B) \xrightarrow{f^{-1}} \mathcal{P}(A)$. Note that, for a theory $\Th$ in coherent logic, there is a one-to-one correspondence between coherent hyperdoctrine morphisms $\F_\Th \to Sub_\Set$ and models of $\Th$ (in $\Sets$). For more background on (coherent) hyperdoctrines the reader is referred to \cite{Law}.

The category of coherent hyperdoctrines and the category of (small) coherent categories are related via an adjunction. More precisely, there is quasi 2-adjunction $\mathcal{A} \colon \pDL \leftrightarrows \Coh \colon \mathcal{S}$. To sketch this adjunction, we first describe the 2-categorical structure in $\Coh$ and in $\pDL$. The 2-cells in $\Coh$ are the natural transformations. To describe the 2-cells in $\pDL$, let $(K, \tau), (K', \tau') \colon P \to P'$ be morphisms in $\pDL$, where $P$ (resp. $P'$) is a coherent hyperdoctrine over $\B$ (resp. $\B'$). A 2-cell $(K, \tau) \to (K', \tau')$ is a natural transformation $\sigma \colon K \to K'$ satisfying, for all $A \in \B$ and $a \in P(A)$, $\tau_A(a) \le P'(\sigma_A)(\tau'_A(a))$. In the remainder we leave the 2-categorical details to the reader interested in those.

As stated above, for a coherent category $\cat{C}$, the functor $Sub_\cat{C} \colon \cat{C}^{\op} \to \DL$ is a coherent hyperdoctrine. This assignment naturally extends to a 2-functor $\S \colon \Coh \to \pDL$, as follows. For a coherent functor $F \colon \C \to \D$, $\S(F) = (F, \tau^F) \colon \S(\C) \to \S(\D)$, where, for $A \in \C$, $\tau_A^F$ is the restriction of $F$ to a map $F_A \colon Sub_{\C}(A) \to Sub_{\D}(FA)$. To ease the notation, we usually write $\S_\C$ for $\S(\C)$, and similarly for morphisms. 

Conversely, for a coherent hyperdoctrine $P$ over $\cat{B}$, we define a coherent category $\mathcal{A}(P)$ whose objects are pairs $(A, a)$, where $A \in \cat{B}$ and $a \in P(A)$. Intuitively, we think of the elements of $P(A)$ as `predicates on $A$' and $(A,a)$ represents $\{x \in A\,|\,a(x)\}$.
A morphism $(A,a) \to (B, b)$ is an element $f \in P(A \times B)$ which is, in the internal language of $P$, a functional relation $\{x\in A\,|\,a(x)\} \to \{y\in B\,|\,b(y)\}$, {\it i.e.}, $f$ is an element of $P(A \times B)$ satisfying
\begin{enumerate}
\item $x:A \,|\,\,\,a(x) \,\vdash\, \exists y:B.\, f(x,y)$;
\item $x:A, y:B \,|\,\,\, f(x,y) \,\vdash\, a(x) \wedge b(y)$;
\item $y,y':B \,|\,\,\,\exists x:A.\,f(x,y) \wedge f(x,y') \,\vdash\, y = y'$,
\end{enumerate} 
where we use the notation from \cite{Jac}. 

\begin{samepage}
Concretely, this comes down to the element $f$ satisfying
\begin{enumerate}
\item $a \le \exists_{\pi_1}(f)$, where $\pi_1 \colon A\times B \to A$;
\item$ f \le \pi_1^*(a) \wedge \pi_2^*(b)$, where $\pi_1 \colon A\times B \to A$ and $\pi_2 \colon A\times B \to B$;
\item $\exists_\pi(\langle \pi_1, \pi_2\rangle^*(f) \wedge \langle \pi_1, \pi_3\rangle^*(f)) \le \exists_{\Delta_B}(\top)$, \\
where $\langle \pi_1, \pi_2\rangle, \langle \pi_1, \pi_3\rangle \colon A \times B \times B \to A \times B$, $\pi \colon A \times B \times B \to B \times B$, $\Delta_B = \langle \id, \id \rangle \colon B \to B \times B$.
\end{enumerate} 
\end{samepage}
For a morphism $(K, \tau) \colon P_1 \to P_2$ of coherent hyperdoctrines over $\B_1$ and $\B_2$, respectively, $\A(K, \tau)$ is given by
$$
\begin{array}{rccl}
\A(K, \tau)\colon & \A(P_1) &\to& \A(P_2) \\
&\xymatrix{(A,a) \ar[d]^-{f}\\(B,b)} &\mapsto & \xymatrix{(K(A), \tau_A(a)) \ar[d]^{\tau_{A\times B}(f) \in P(K(A\times B)) \cong P(K(A) \times K(B))}\\(K(B), \tau_B(b))}.
\end{array}
$$
This yields a 2-functor $\A \colon \pDL \to \Coh$. The 2-functors $\A$ and $\S$ form an adjunction, but the diagram expressing the naturality of the counit commutes \emph{up to isomorphism}. In the literature, various names are used for such an adjunction. In \cite{Gray74} it is called a quasi 2-categorical adjunction, while in \cite{Bor1} it is just called a 2-categorical adjunction. The following proposition is proven, in a slightly different form, in \cite{Pi83}.

\begin{proposition}
\label{prop:adjpdlcoh}
The 2-functors $\A \colon \pDL \leftrightarrows \Coh \colon \S$ form a (quasi) 2-categorical adjunction, $\A \dashv \S$, and, for each $\C \in \Coh$, the counit  at $\C$, $\epsilon_\C \colon \A(\S(\C)) \to \C$, is an equivalence in the 2-category $\Coh$ (as in Definition 7.1.2 in \cite{Bor1}).
\end{proposition}

As we will show in the next proposition, coherent hyperdoctrines are closed under canonical extension, in the sense that, for a coherent hyperdoctrine $P \colon \cat{B}^{\op} \to \DL$, $\CE \after P \colon \B^{\op} \to \DL$ is again a coherent hyperdoctrine. 
In combination with the 2-adjunction of Proposition~\ref{prop:adjpdlcoh} this yields a natural notion of canonical extension for coherent categories. 

\begin{proposition}
Let $P \colon \cat{B}^{\op} \to \DL$ be a coherent hyperdoctrine. The functor 
$
P^\delta = \CE \after P \colon \B^{\op} \to \DL
$ 
is again a coherent hyperdoctrine and the assignment $P \mapsto P^\delta$ extends to a functor $\pDL \to \pDL$. Furthermore, the morphism $(\id, \eta^P) \colon P \to P^\delta$, where $\eta^P$ is given by, for $A \in \cat{B}$,
$$
\eta_A^P = e_{P(A)} \colon P(A) \to P(A)^\delta = P^\delta(A),
$$
is a morphism of coherent hyperdoctrines.
\end{proposition}

\begin{proof}
Let $P \colon \cat{B}^{\op} \to \DL$ be a coherent hyperdoctrine. Adjunctions are preserved under canonical extension, see {\it e.g.} Proposition 3.6 in \cite{DGP}. Hence, for $A \xrightarrow{\alpha} B$ in $\cat{B}$, $(\exists_\alpha)^\delta$ is left adjoint to $P^\delta(\alpha) = (P(\alpha))^\delta$. To prove that $P^\delta$ is a coherent hyperdoctrine, we show that the Beck-Chevalley condition (BC) and Frobenius reciprocity (F) are canonical, that is, if $P$ satisfies (BC) (resp. (F)), then so does $P^\delta$. 
First we consider (BC). Let the following diagram be a pullback in $\cat{B}$.
$$
\xymatrix{
Q \ar[r]^-{\alpha'} \ar[d]_-{\beta'} & B \ar[d]^-{\beta}\\
A \ar[r]_-{\alpha} & C
}
$$
Then
$$
\begin{array}{rcll}
P(\alpha)^\delta \after \exists_{\beta}^\delta
&=&
(P(\alpha) \after \exists_{\beta})^\delta &\text{(Proposition~\ref{prop:prescomp})}\\
&=&
(\exists_{\beta'} \after P(\alpha'))^\delta &\text{($P$ satisfies (BC))}\\
&=&
\exists_{\beta'}^\delta \after P(\alpha')^\delta &\text{(Proposition~\ref{prop:prescomp})}.
\end{array}
$$ 
To prove canonicity of the Frobenius reciprocity, note that this condition (on $P$) may be formulated by saying that, for all $A \xrightarrow{\alpha} B$ in $\B$, the following diagram commutes.
$$
\xymatrix@C=4pc{
P(A) \times P(B) \ar[r]^-{\id \times P(\alpha)} \ar[d]_-{\exists_{\alpha}\times\id} & P(A) \times P(A) \ar[r]^-{\wedge} & P(A) \ar[d]^-{\exists_{\alpha}}\\
P(B) \times P(B) \ar[rr]_-{\wedge} && P(B)
}
$$
Note that the meet operation $\wedge\colon L \times L \to L$ on a lattice $L$ is meet preserving and therefore has a unique extension $\wedge^\delta \colon L^\delta \times L^\delta \cong (L \times L)^\delta \to L^\delta$. As meet is right adjoint to the diagonal map and adjunctions are preserved under canonical extension, the map $\wedge^\delta$ gives the meet on $L^\delta$. As above for (BC), we may use Proposition~\ref{prop:prescomp} to derive that also in $P^\delta$ the Frobenius reciprocity holds.

For a morphism $(K, \tau) \colon P_1 \to P_2$ of coherent hyperdoctrines over $\B_1$ and $\B_2$, we define a morphism $(K, \tau^\delta) \colon P_1^\delta \to P_2^\delta$ by, for $A \in \B_1$, $\tau^\delta_A = (\tau_A)^\delta \colon P_1(A)^\delta \to P_2(K(A))^\delta$. To prove that $\tau^\delta$ is a natural transformation which preserves existential quantification, one may again rely on Proposition~\ref{prop:prescomp} and use the fact that $\tau$ has these properties. 

It is readily checked that $(\id, \eta^P) \colon P \to P^\delta$ is a morphism of coherent hyperdoctrines.\QED
\end{proof}

In particular, for a coherent category $\C$, $\S_\C^\delta$ is a coherent hyperdoctrine. This leads us to study the following construction.

\begin{definition}
The \emph{canonical extension of a coherent category} $\C$ is defined to be the functor $E_\C \colon \C \to \A(\S_\C^\delta)$ given by
$$
\begin{array}{ccccl}
\C &\xrightarrow{\sim}& \A(\S_{\C}) &\xrightarrow{\A(\id, \eta^{\S_\C})} &\A({\S_{\C}}^\delta) \\
\xymatrix{A \ar[d]^-{\alpha}\\B} &\mapsto & \xymatrix{(A, \top) \ar[d]^{\langle \id, \alpha\rangle}\\(B,\top)} & \mapsto & \xymatrix{(A, \top) \ar[d]^{\langle \id, \alpha\rangle \in \S_\C(A \times B) \hookrightarrow \S_\C^\delta(A\times B)}\\(B, \top)}.
\end{array}
$$ 
and we write $\C^\delta = \A(\S_\C^\delta)$. The assignment $\C \mapsto \C^\delta$ extends to a 2-functor on the category $\Coh$.
\end{definition}

This definition extends the existing notion of canonical extension for distributive lattices in the sense that, for a distributive lattice $\L$ (viewed as a coherent category), the category $\A(\S_\L^\delta)$ is equivalent to the (ordinary) canonical extension $\L^\delta$ of $\L$. To prove this we use the following fact.

\begin{lemma}
For a distributive lattice $\L$ and $a \in \L$, the canonical extension $(\downarrow_\L \!a)^\delta$ of the downset of $a$ in $\L$ is the restriction of the embedding $\L \hookrightarrow \L^\delta$ to $\downarrow_\L \!a \hookrightarrow \,\downarrow_{\L^\delta}\! a$.
\end{lemma} 

\auxproof{
Clearly $\downarrow_{\L^\delta}\! a$ is a complete lattice. Hence, it remains to show that the given embedding is dense and compact. We start by proving denseness. Let $u \in \downarrow_{\L^\delta\!} a$. Then
$$
\begin{array}{rcll}
u &=& \bigwedge \{y \in \Op{L}\,|\,u \le y\} &\text{(denseness of $\L \hookrightarrow \L^\delta$)}\\
&=&
a \wedge \bigwedge \{y \in \Op{L}\,|\,u \le y\} &\text{($u \le a$)}\\
&=&
\bigwedge \{y \wedge a\,|\,y \in \Op{L}\,|\,u \le y\} &\text{(distributivity)}\\
&=&
\bigwedge \{y \in \Op{L}\,|\,u \le y \le a\} &\text{(for $y \in \Op{L}$, $y \wedge a \in \Op{L}$)}\\
&=&
\bigwedge \{y \in I(\downarrow_{\L^\delta}\! a)\,|\,u \le y\} &\text{($I(\downarrow_{\L^\delta}\! a) = \Op{L} \,\cap \,\downarrow_{\L^\delta}\! a$)}\\
\end{array}
$$
This proves $\bigwedge\bigvee$-denseness. The $\bigvee\bigwedge$-denseness follows immediately from the fact that all filter elements of $\L^\delta$ below $a$ are also filter elements in $\downarrow_{\L^\delta}\! a$. Then, for $u \in \downarrow_{\L^\delta}\! a$, 
$$
\begin{array}{rcll}
u &=& \bigvee\{x \in \Cl{L}\,|\,x \le u\}\\
&=&
\bigvee\{x \in F(\downarrow_{\L^\delta}\! a)\,|\,x \le u\} &\text{(as $x \le u$ implies $x \le a$)}.
\end{array}
$$
To prove compactness, let $F, I \subseteq \downarrow_\L\! a$ s.t. $\bigwedge F \le_{\downarrow_{\L^\delta}\! a} \bigvee I$. Then also $\bigwedge F \le_{\L^\delta} \bigvee I$. Hence there exist finite subsets $F' \subseteq F$, $I' \subseteq I$ such that $\bigwedge F' \le_\L \bigvee I'$. This inequality clearly also holds in $\downarrow_\L\! a$.\QED
}

\begin{proposition}
\label{prop:dl}
Let $\L$ be a distributive lattice. Viewing $\L$ as a coherent category, $\A(\S_\L^\delta) \simeq \L^\delta$.
\end{proposition}

\begin{proof}
To define an equivalence $L^\delta \to \A(\S_\L^\delta)$, we start by describing the category $\A(\S_\L^\delta)$. For a distributive lattice $\L$ and $a \in \L$, $\S_\L(a)$ is the downset $\downarrow_\L\!a$ of $a$ in $\L$.
Hence, the functor $\S_\L^\delta$ is given by
$$
\begin{array}{rcl}
\S_\L^\delta \colon \L^{\op} &\to& \DL\\
a &\mapsto& (\downarrow_\L \!a)^\delta = \downarrow_{\L^\delta}\! a\\
a \le b &\mapsto&
\begin{array}[t]{rcl}
 \downarrow_{\L^\delta}\! b &\to & \downarrow_{\L^\delta}\! a\\
u &\mapsto & u \wedge a.
\end{array}
\end{array}
$$
It follows that the objects of $\A(\S_\L^\delta)$ are pairs $(a, u)$ where $a \in \L$ and $u \in \downarrow_{\L^\delta}\! a$. The essence of the proof is that, for $u \in L^\delta$, all `copies $(a,u)$ of $u$' in $\A(\S_L^\delta)$ are isomorphic. To give a precise proof, we first describe the morphisms in $\A(\S_\L^\delta)$.

Note that the left adjoint to $\S_\L^\delta(a\le b)$ is just the inclusion map $\downarrow_{\L^\delta}\! a \hookrightarrow \downarrow_{\L^\delta}\! b$. By definition, a morphism $(a, u) \to (b, v)$ is an element $w$ of $\S_\L^\delta(a \times b) = \downarrow_{\L^\delta}\! (a \wedge b)$ satisfying
\begin{enumerate}
\item $u \le \exists_{\pi_1}(w) = w$;
\item$ w \le \pi_1^*(u) \wedge \pi_2^*(v) = u \wedge (a \wedge b) \wedge v \wedge (a \wedge b) = u \wedge v$;
\item $\exists_\pi(\langle \pi_1, \pi_2\rangle^*(w) \wedge \langle \pi_1, \pi_3\rangle^*(w)) \le \exists_{\Delta_J}(\top) = \top$, this is vacuously true.
\end{enumerate}
Hence, the only candidate for a morphism  $(a, u) \to (b, v)$ is $u$. It follows that there is a (unique) morphism  $(a, u) \to (b, v)$ if and only if $u \le v$. Therefore, for $(a,u) \in \A(\S_\L^\delta)$, $(a, u) \cong (1, u)$.  Define a functor
$$
\begin{array}{rcll}
\L^\delta &\to& \A(\S_\L)^\delta\\
u &\mapsto & (1,u)\\
u \le v &\mapsto &(1,u) \xrightarrow{u} (1,v).
\end{array}
$$
This functor is full and faithful and essentially surjective, {\it i.e.}, it is an equivalence. \QED
\end{proof}

Our construction of canonical extension for coherent categories has a universal property similar to the one known from the algebraic setting, as we prove in Theorem~\ref{thm:extcohcat}. It is slightly more complicated than the universal property for distributive lattices because of the existential quantifiers (the left adjoints). These are similar to the diamond operators from modal logic (both are join-preserving maps) and the Esakia Lemma plays a crucial role.

\begin{lemma}[Esakia Lemma, \cite{Es}]
Let $f \colon \L \to \K$ be a join-preserving map between distributive lattices. For any filtered subset $F \subseteq \Cl{L}$, $f^\delta(\bigwedge F) = \bigwedge f^\delta[F]$.
\end{lemma}

The following definition generalises the notion of $\DL^+$ to the categorical setting.

\begin{definition}
We define $\Coh^+$ to be the category of coherent categories with the additional property that all subobjects lattices are completely distributive algebraic and the pullback functors preserve all joins. The morphisms in $\Coh^+$ are coherent functors which preserve all meets and joins of subobjects.
\end{definition}

Note that, as pullback morphisms preserve all limits, in $\Coh^+$ the pullback morphisms are complete lattice homomorphisms.  

The mapping $\C \mapsto \C^\delta$ is \emph{not} left adjoint to the inclusion functor $\Coh^+ \hookrightarrow \Coh$. For a coherent functor $M \colon \C \to \D$, where $\C \in \Coh$ and $\D \in \Coh^+$, there is, up to isomorphism, a unique extension of $M$ to a functor $\widetilde{M} \colon \C^\delta \to \D$, which preserves all meets and joins of subobjects. However, this functor $\widetilde{M}$ does not preserve existential quantification in general. Therefore, to describe a universal property of our construction, we have to restrict the morphisms we consider. This leads to the definition of a p-model. A similar notion was introduced by Makkai in \cite{Mak} to describe the universal property of his topos of types construction. 

\begin{definition}
\label{def:pmod}
Let $\C$ be a coherent category, $\D \in \Coh^+$ and $M \colon \C \to \D$ a coherent functor. We say $M$ is a \emph{p-model} iff, for all $A \xrightarrow{\alpha} B$ in $\C$ and $\rho$ a prime filter in $Sub_\C(A)$,
$$
\exists_{M(\alpha)}(\bigwedge\{M(U)\,|\,U \in \rho\}) = \bigwedge\{\exists_{M(\alpha)}(M(U))\,|\, U \in \rho\},
$$
where the meets are taken in $Sub_\D(M(A))$ and $Sub_\D(\exists_{M(\alpha)}(M(A)))$, respectively.
\end{definition}

\begin{proposition}
For a coherent category $\C$, $E_\C \colon \C \to \C^\delta$ is a p-model.
\end{proposition}

\begin{proof}
Let $\C$ be a coherent category. First we show that $\C^\delta = \A(\S_\C^\delta)$ is in $\Coh^+$. For $(A, u) \in \A(\S_\C^\delta)$, $Sub(A,u)$ is isomorphic to the downset $\downarrow\!\!u$ in $\S_\C^\delta(A)$. As $\S_\C^\delta(A) = Sub_\C(A)^\delta$ is the canonical extension of a distributive lattice, it is completely distributive algebraic. Therefore, also $Sub(A,u)$ is completely distributive algebraic. 

For $(A,u) \xrightarrow{f} (B,v)$ in $\A(\S_{\C}^\delta)$, consider the pullback functor 
$$
f^* \colon Sub(B, v) \cong \, \downarrow_{\S_\C^\delta(B)}\! v \,\to\,\, \downarrow_{\S_\C^\delta(A)}\! u \cong Sub(A,u)
$$ 
We have to show that the map $f^*$ preserves arbitrary joins. Recall that $f \in \S_\C^\delta(A \times B)$ is, in the internal language of $\S_\C^\delta$, a functional relation $\{x \in A\,|\,u(x)\} \to \{y \in B\,|\,v(x)\}$. In this language, for $w \in \downarrow_{\S_\C^\delta(B)} v$, its inverse image under $f$ may be described as
$$
\begin{array}{rcl}
f^{-1}(w) &=& \{x \in A\,|\,\exists y \in B.\,f(x,y) \wedge w(y)\}.
\end{array}
$$
Hence, the pullback functor $f^*$ is given by, for $w \in \downarrow_{\S_\C^\delta(B)} v$,
$$
\begin{array}{rcl}
f^*(w) &=& \exists_{\pi_1}^{\S_{\C}^\delta}(f \wedge \S_{\C}^\delta(\pi_2)(w)) 
\end{array}
$$
As $\exists_{\pi_1}^{\S_\C^\delta}$ is a left adjoint, it preserves all joins. Also $\S_\C^\delta(\pi_2)$ preserves all joins (being the extension of a join-preserving map). Using the fact that $\S_\C^\delta(A \times B)$ is completely distributive, it now follows that $f^*$ preserves all joins. 

To show that $E_\C$ is a p-model, let $A \xrightarrow{\alpha} B$ in $\C$ and $\rho$ a prime filter in $Sub_\C(A)$. We have to show
\begin{equation}
\label{eq:fmM0}
\exists_{E_\C(\alpha)}(\bigwedge\{E_\C(U)\,|\,U \in \rho\}) = \bigwedge\{\exists_{E_\C(\alpha)}(E_\C(U))\,|\, U \in \rho\}.
\end{equation}
Recall that $Sub_{\A(\S_\C^\delta)}(A, \top) \cong \S_\C^\delta(A) = \S_\C(A)^\delta$. For $m \colon U \hookrightarrow A$ in $\S_\C(A)$, 
$$E_\C(U) = (U, \top) \cong (A,U),$$ 
where the isomorphism is given by ($\langle \id, m\rangle \colon U \hookrightarrow U \times A) \in \S_\C(U \times A) \subseteq \S_\C(U \times A)^\delta$. 
So $E_\C$ sends subobjects of $A$ to their image under the embedding $\S_\C(A) \hookrightarrow \S_\C^\delta(A) \cong Sub_{\A(\S_\C^\delta)}(A, \top)$ and we may identify $E_\C(U)$ with $U$. Equation \eqref{eq:fmM0} then comes down to
\begin{equation}
\label{eq:fmM0-2}
\exists_{E_\C(\alpha)}(\bigwedge\{U\,|\,U \in \rho\}) = \bigwedge\{\exists_{E_\C(\alpha)}(U)\,|\, U \in \rho\}. 
\end{equation}
We will show that $\exists_{E_\C(\alpha)}$ is the canonical extension of $\exists_\alpha \colon \S_\C(A) \to \S_\C(B)$ so that \eqref{eq:fmM0-2} follows from the Esakia Lemma.  
Since left adjoints are unique, it suffices to show that the pullback morphism 
$$E_\C(\alpha)^* \colon Sub_{\A(\S_\C^\delta)}(B, \top) \cong \S_C(B)^\delta \to \S_\C(A)^\delta \cong Sub_{\A(\S_\C^\delta)}(A, \top)$$ 
is the canonical extension of $\alpha^* \colon \S_C(B) \to \S_C(A)$. As both $E_\C(\alpha)^*$ and $(\alpha^*)^\delta$ are complete homomorphisms and $\S_\C(B)$ is dense in $\S_\C(B)^\delta$, it suffices to consider elements from $\S_\C(B)$. For all $V \in \S_\C(B)$,
$$
\begin{array}{rcll}
E_\C(\alpha)^*(V) &=&
\langle \id, \alpha \rangle^*(V)\\ 
&=&
\exists_{\pi_1}^{\S_\C^\delta}(\langle \id, \alpha \rangle \wedge \S_\C^\delta(\pi_2)(V))\\
&=&
\exists_{\pi_1}^{\S_\C}(\langle \id, \alpha \rangle \wedge \S_\C(\pi_2)(V))&\text{($V \in \S_\C(B)$)}\\
&=&
\alpha^*(V)\\
&=&
(\alpha^*)^\delta(V),
\end{array}
$$
where, in the second equality, we use the description of pullback morphisms in $\A(\S_\C^\delta)$ given above.
Hence $E_\C(\alpha)^* = (\alpha^*)^\delta$ and therefore $\exists_{E_\C(\alpha)} = \exists_\alpha^\delta$ and the claim follows from the Esakia lemma.\QED
\auxproof{
For the fourth equality note that
$$
\begin{array}{rcll}
\exists_{\pi_1}(\langle \id, \alpha \rangle \wedge \S_\C(\pi_2)(V))
&=&
\exists_{\pi_1}((\alpha \times \id)^*(\exists_{\Delta_B}(\top)) \wedge \pi_2^*(V)) &(\pi_2 \colon I\times J \to J)\\
&=&
\exists_{\pi_1}((\alpha \times \id)^*(\exists_{\Delta_B}(\top)) \wedge (\alpha\times\id)^*(\pi_2^*(V)))&(\pi_2 \colon J \times J \to J)\\
&=&
\exists_{\pi_1}((\alpha \times \id)^*(\exists_{\Delta_B}(\top) \wedge \pi_2^*(V)))\\
&=&
\alpha^*(\exists_{\pi_1}(\exists_{\Delta_B}(\top) \wedge \pi_2^*(V))) &\text{(BC)}\\
&=&
\alpha^*(\exists_{\pi_1}(\exists_{\Delta_B}(\top \wedge \Delta_B^*(\pi_2^*(V))))) &\text{(Frobenius)}\\
&=&
\alpha^*(V).
\end{array}
$$
}
\end{proof}

We are now ready to describe a universal property of our notion of canonical extension for coherent categories. Let $\C$ be a coherent category and $\D \in \Coh^+$. For a p-model $M \colon \C \to \D$, there exists a morphism $\widetilde{M} \colon \C^\delta \to \D$ in $\Coh^+$ such that the following diagram commutes.
$$
\xymatrix{
\C \ar[r]^-{E_\C} \ar[rd]_-{M} & \C^\delta \ar[d]^-{\widetilde{M}}\\
&\D
}
$$
The morphism $\widetilde{M}$ is unique up to a natural isomorphism. To make this precise in the next theorem we have to work in a 2-categorical setting. We first introduce some notation. For $\C \in \Coh$ and $\D \in \Coh^+$, we write $\Coh_p(\C, \D)$ for the category of p-models of $\C$ in $\D$ with natural transformations and $\Coh^+(\C^{\delta}, \D)$ for the category of morphism $\C^\delta \to \D$ in $\Coh^+$ with natural transformations. 

\begin{theorem}
\label{thm:extcohcat}
Let $\C$ be a coherent category and $\D \in \Coh^+$. Precomposition with the functor $E_\C \colon \C \to \C^{\delta}$ yields an equivalence 
$$
\mathcal{F} = \_ \after E_\C \colon \Coh^+(\C^\delta, \D) \to \Coh_p(\C, \D)
$$
in the 2-category $\cat{Cat}$, of (small) categories.
\end{theorem}

\begin{proof}
We define a functor 
$$
\mathcal{G} \colon \Coh_p(\C, \D) \to \Coh^+(\C^\delta, \D).
$$
Let $M\colon \C \to \D$ be a p-model. We rely on Proposition~\ref{prop:adjpdlcoh} and the natural isomorphism in \eqref{eq:natdl} to define a morphism $\mathcal{G}(M) \colon \C^\delta \to \D$. Consider $\S_M \colon \S_\C \to \S_\D$. Recall that $\S_M = (M, \tau)$, where, for $A \in \C$, $\tau_A \colon \S_\C(A) \to \S_\D(M(A))$ is the restriction of $M$ to $\S_\C(A)$. Consider $(M, \overline{\tau}) \colon \S_\C^\delta \to \S_\D$, where, for $A \in \C$, $\overline{\tau}_A = \overline{\tau_A} \colon \S_\C(A)^\delta \to \S_\D(M(A))$, the unique extension of $\tau_A$ to a complete lattice homomorphism. We first show that $(M, \overline{\tau})$ is a morphism of coherent hyperdoctrines. Clearly $M$ preserves limits (as it is a coherent functor) and, for each $A \in \C$, $\overline{\tau}_A$ is a (complete) lattice homomorphism.
To prove naturality of $\overline{\tau}$, let $A \xrightarrow{\alpha} B$ in $\C$. Then
$$
\begin{array}{rcll}
\overline{\tau}_A \after \S_\C(\alpha)^\delta 
&=&
\overline{\tau_A \after \S_\C(\alpha)} &\text{(naturality of $\overline{(\_)}$)}\\
&=&
\overline{\S_\D (M(\alpha)) \circ \tau_B} &\text{(naturality of $\tau$)}\\
&=&
\S_\D(M(\alpha)) \circ \overline{\tau}_B &\text{(naturality of $\overline{(\_)}$)}.
\end{array}
$$

Finally, to prove that $\overline{\tau}$ preserves existential quantification, let $A \xrightarrow{\alpha} B$ in $\C$ and consider
$$
\xymatrix{
\S_\C(A) \ar[r]^-{\exists_\alpha} \ar[d]_-{\tau_A} & \S_\C(B) \ar[d]^-{\tau_B}\\
\S_\D(MA) \ar[r]_-{\exists_{M\alpha}} & \S_\D(MB)
}
$$
As $M$ is a p-model we have, for every prime filter $\rho$ in $Sub_\C(A)$,
$$
\exists_{M(\alpha)}(\bigwedge\{M(U)\,|\,U \in \rho\}) = \bigwedge\{\exists_{M(\alpha)}(M(U))\,|\, U \in \rho\}.
$$
Hence, Proposition~\ref{prop:comjpm} applies and we may conclude $\overline{\tau}_B \after (\exists_{\alpha})^\delta = \exists_{M\alpha} \after \overline{\tau}_A$.

We now define $\mathcal{G}(M) \colon \A(\S_{\C}^\delta) \to \D$ to be the composite
$$
\A(\S_{\C}^\delta) \xrightarrow{\A(M, \overline{\tau})} \A(\S(\D)) \xrightarrow{\epsilon_\D} \D,
$$
where $\epsilon_\D$ is the counit of the adjunction described in Proposition~\ref{prop:adjpdlcoh}.
As, for each $A \in \C$, $\overline\tau_A$ is a complete lattice homomorphism, it follows that $\mathcal{G}(M)$ preserves arbitrary meets and joins of subobjects. Hence, it is a morphism in $\Coh^+$. We leave it to the reader to define the action of $\mathcal{G}$ on morphisms. One may show that $\mathcal{F} \after \mathcal{G} = \id$ and there is a natural isomorphism $\id \to \mathcal{G} \after \mathcal{F}$. Hence, $\mathcal{F}$ is a 2-equivalence.\QED
\end{proof}

It follows that the canonical extension of a coherent category is determined uniquely, up to equivalence in the 2-category $\Coh^+$, by the universal property described in the theorem above.
Remark that, for a distributive lattice $\L$ and $K \in \DL^+$ (both viewed as categories), a p-model $\L \to \K$ is just a lattice homomorphism and the above theorem generalises Theorem~\ref{thm:extlat}.

\section{Canonical extension for Heyting categories}
\label{sec:Heyt}
For a distributive lattice $\L$, its canonical extension is complete and completely distributive, hence it is in particular a Heyting algebra. In case $\L$ itself is already a Heyting algebra, the embedding $e_\L \colon \L \to \L^\delta$ preseves the Heyting implication. Furthermore, the canonical extension of a morphism of Heyting algebras is again a morphism of Heyting algebras. Hence, canonical extension yields a functor on the category of Heyting algebras. In this section we explain how these results lift to the categorical setting. For more background on canonical extension in the setting of Heyting algebras, see \cite{Geh2012}. 

The categorical analogue of a Heyting algebra is a \emph{Heyting category}, {\it i.e.}, a coherent category $\C$ such that, for all $A \xrightarrow{\alpha} B$ in $\C$, the pullback functor $\alpha^* \colon Sub_\C(B) \to Sub_\C(A)$ has a right adjoint $\forall_\alpha$, called \emph{universal quantification along $\alpha$}. A morphism of Heyting categories is a coherent functor which also preserves universal quantification. We write $\Heyt$ for the category of Heyting categories. Heyting categories provide sound and complete semantics for intuitionistic first order logic.

Note that in a Heyting category all subobject lattices are Heyting algebras. For let $\C$ be a Heyting category, $A \in \C$ and $m \colon U \hookrightarrow A \in Sub_\C$. The pullback functor $m^* \colon Sub_\C(A) \to Sub_\C(U)$ is given by $m^*(V) = U \wedge V$. It follows that, for $W \in Sub_\C(A)$, $U \to W = \forall_m(m^*(W))$. Morphisms between Heyting categories preserve this implication.

For a coherent category $\C$, the pullback functors in $\C^\delta$ are complete homomorphisms and therefore they have right adjoints. Hence, the canonical extension of a coherent category is a Heyting category.

To study the properties of the canonical extension of Heyting categories, we rely on the 2-adjunction between coherent categories and coherent hyperdoctrines of Proposition~\ref{prop:adjpdlcoh}. This 2-adjunction restricts to a 2-adjunction between Heyting categories and first order hyperdoctrines.

\begin{definition}
A \emph{first order hyperdoctrine} is a coherent hyperdoctrine $P \colon \B^{\op} \to \DL$ such that, for all $A \in \B$, $P(A)$ is a Heyting algebra and, for all $A \xrightarrow{\alpha} B$, $P(\alpha) \colon P(B) \to P(A)$ has a right adjoint $\forall_\alpha$. A morphism between first order hyperdoctrines is a morphism of coherent hyperdoctrines which, in addition, preserves implication and universal quantification. We write $\pHA$ for the category of first order hyperdoctrines.  
\end{definition}

Remark that the fact that a first order hyperdoctrine $P \colon \B^{\op} \to \DL$ satisfies Frobenius implies that, for all $A \xrightarrow{\alpha} B$ in $\B$, $P(\alpha) \colon P(B) \to P(A)$ preserves the Heyting implication. Hence, a first order hyperdoctrine may be viewed as a functor $\B^{\op} \to \cat{HA}$. 

\auxproof{
Let $A \xrightarrow{\alpha} B$ in $\B$ and $a, b \in P(B)$. We have to show $P(\alpha)(a \to b) = P(\alpha)(a) \to P(\alpha)(b)$. First note that
$$
P(\alpha)(a \to b) \le P(\alpha)(a) \to P(\alpha)(b) \quad\text{iff}\quad P(\alpha)(a) \wedge P(\alpha)(a \to b) \le P(\alpha)(b).
$$
As $P(\alpha)$ preserves meets, the right inequality clearly holds.

To prove the converse inquality we apply the adjunction $\exists_{\alpha} \dashv P(\alpha)$:
$$
\begin{array}{rcl}
P(\alpha)(a) \to P(\alpha)(b) \le P(\alpha)(a \to b) &\text{iff}&
\exists_{\alpha}(P(\alpha)(a) \to P(\alpha)(b)) \le a \to b\\
&\text{iff}&
a \wedge \exists_{\alpha}(P(\alpha)(a) \to P(\alpha)(b)) \le b\\
&\text{iff}&
\exists_{\alpha}(P(\alpha)(a) \wedge (P(\alpha)(a) \to P(\alpha)(b))) \le b.
\end{array}
$$
The claim now follows from the fact that $\exists_{\alpha} \after P(\alpha) \le \id$.
}

\begin{proposition}
The 2-adjunction $\A \colon \pDL \leftrightarrows \Coh \colon \S$ of Proposition~\ref{prop:adjpdlcoh} restricts to a 2-adjunction $\A \colon \pHA \leftrightarrows \Heyt \colon \S$.
\end{proposition}

\begin{proof}
It is clear that, for a Heyting category $\C$, $\S_\C$ is a first order hyperdoctrine. Let $F \colon \C \to \D$ be a morphism of Heyting categories. We have to show that, for $A \in \C$, the restriction of $F$ to $\S_\C(A) \to \S_\D(FA)$ preserves the implication. For $U, W \in \S_\C(A)$, where $U$ may be represented by $m \colon U \hookrightarrow A$,
$$
F(U \to W)
=
F(\forall_m(m^*(W))) 
=
\forall_{F(m)}(F(m)^*(F(W))) = F(U) \to F(W),
$$
as required. It is clear that $\S_F$ preserves universal quantification, whence it is a morphism in $\pHA$.

Conversely, let $P$ be a first order hyperdoctrine. For a morphism $(A, u) \xrightarrow{f} (B,v)$ in $\A(P)$, the pullback functor has a right adjoint given by, for $w \in \,\downarrow_{P(A)}\! u \,\cong\, Sub_{\A(P)}(A,u)$,
$$
\forall_f(w) = \forall_{\pi_2}^P(f \to P(\pi_1)(w)).
$$
So $\A(P)$ is a Heyting category. Let $(K, \tau) \colon P_1 \to P_2$ be a morphism in $\pHA$ and $(A, u) \xrightarrow{f} (B,v)$ in $\A(P)$. We have to show 
$$
\tau_B \after \forall_f = \forall_{\tau_{A\times B}(f)} \after \tau_A \colon Sub_{\A(P)}(A, u) \to Sub_{\A(Q)}(KA, \tau_B(v)).
$$ 
Let $w \in \,\downarrow_{P(A)} u \cong Sub_{\A(P)}(A,u)$. Then
$$
\begin{array}{rcll}
\tau_B(\forall_f(w))
&=&
\tau_B(\forall_{\pi_2}^{P_1}(f \to P_1(\pi_1)(w)))\\
&=&
\forall_{K\pi_2}^{P_2}(\tau_{A\times B}(f \to P_1(\pi_1)(w)))\\
&=&
\forall_{K\pi_2}^{P_2}(\tau_{A\times B}(f) \to \tau_{A\times B}(P_1(\pi_1)(w)))\\
&=&
\forall_{K\pi_2}^{P_2}(\tau_{A\times B}(f) \to P_2(\pi_1)(\tau_A(w)))\\
&=&
\forall_{\tau_{A \times B}(f)}(\tau_A(w)).
\end{array}
$$
Hence, $\A(K, \tau)$ is a morphism of Heyting categories, which completes the proof.\QED
\end{proof}

We now show that canonical extension preserves the Heyting structure.

\begin{proposition}
Canonical extension of coherent hyperdoctrines $(\_)^\delta \colon \pDL \to \pDL$ restricts to a 2-functor on the 2-category of first order hyperdoctrines. Furthermore, for a first order hyperdoctrine $P$, the embedding $(\id, \eta^P) \colon P \to P^\delta$ preserves the Heyting structure. 
\end{proposition}

\begin{proof}
As the canonical extension of a Heyting algebra is again a Heyting algebra and canonical extension preserves adjunctions, it readily follows that the canonical extension $P^\delta$ of a first order hyperdoctrine $P$ is again a first order hyperdoctrine.

Now let $(K, \tau) \colon P_1 \to P_2$ be a morphism of first order hyperdoctrines over $\B_1$ and $\B_2$ respectively. As, for each $A \in \B_1$, $\tau_A$ preserves implication, so does $\tau_A^\delta$. It is left to show that $\tau^\delta$ preserves universal quantification. Let $A \xrightarrow{\alpha} B$ in $\B_1$. We have to show that $\tau_B^\delta \after \forall_{\alpha}^\delta = \forall_{K\alpha}^\delta \after \tau_A^\delta$. This follows from Proposition~\ref{prop:prescomp} and the fact that $\tau_B \after \forall_{\alpha} = \forall_{K\alpha} \after \tau_A$. Hence, $(K, \tau^\delta)$ is a morphism of first order hyperdoctrines.

It is readily checked that $(\id, \eta^P) \colon P \to P^\delta$ is a morphism of first order hyperdoctrines.\QED 
\end{proof}

From the above two propositions we deduce the following. 

\begin{corollary}
The canonical extension functor $(\_)^\delta \colon \Coh \to \Coh$ restricts to a 2-functor on the 2-category of Heyting categories. Furthermore, for a Heyting category $\C$, $E_\C \colon \C \to \C^\delta$ is a morphism of Heyting categories.
\end{corollary}

\section{Topos of types}
\label{sec:tt}
In \cite{Mak} Makkai defines, for a coherent category $\C$, its topos of types $T(\C)$. He applies this construction to show the existence of full embeddings of certain coherent toposes into functor categories. Furthermore, he views the topos of types as a conceptual tool which enables one to formulate precisely certain natural intuitive questions from model theory. 

For a coherent category $\C$, the coherent hyperdoctrine $\S_\C^\delta$, as defined in the previous section, is an internal locale in $\Set^{\C^{op}}$ (actually even in $Sh(\C, J_{coh})$, the category of sheaves over $\C$ with the coherent topology). In this section we prove that $T(\C)$ is equivalent to the topos of sheaves over this internal locale $\S_\C^\delta$ (Theorem~\ref{thm:inttt} and Theorem~\ref{thm:shjcoh}). We apply our alternative description in Section~\ref{sec:proptt}, where we study some properties of the topos of types construction.  

Let $\C$ be a coherent category. To define the topos of types $T(\C$), Makkai starts from the \emph{category of filters} $\Lambda\C$ of $\C$.  The objects of $\Lambda\C$ are pairs $(A,F)$, where $A \in \C$ and $F$ is a filter in $Sub_\C(A)$. A morphism $(A, F) \to (B,G)$ is a germ $[\alpha]$ of so-called `local continuous maps'. A local continuous map $\alpha \colon (A,F) \to (B,G)$ is a morphism $U \xrightarrow{\alpha} B$ in $\C$, where $U \in F$, such that, for all $V \in G$, $\alpha^*(V) \in F$. Two such maps $\alpha_1 \colon U_1 \to B$, $\alpha_2 \colon U_2 \to B$ are equivalent if and only if there exists $U \in F$ with $U \le U_1 \wedge U_2$ and $\alpha_1\!\upharpoonright U = \alpha_2\!\upharpoonright U$.    

Using the adjunction $\A \colon \pDL \leftrightarrows \Coh \colon \S$ of Proposition \ref{prop:adjpdlcoh}, one may give an alternative description of the category of filters of $\C$, relating it to the construction of the lattice of filters of a distributive lattice. Let $Fl \colon \DL \to \DL$ be the functor which sends a distributive lattice $\L$ to the lattice $Fl(\L)$ of filters of $\L$ ordered by reverse inclusion. For a morphism $f \colon \L \to \K$, $Fl(f) \colon Fl(\L) \to Fl(\K)$, sends a filter $F$ in $\L$ to the filter in $\K$ generated by the direct image of $F$, {\it i.e.},
$$
Fl(f)(F) = \,\uparrow\!\!\{f(a)\,|\, a \in F\}.
$$
For a coherent hyperdoctrine $P$, $Fl \after P$ is again a coherent hyperdoctrine \cite{Pi83}. One may show that, for a coherent category $\C$, the category of filters $\Lambda\C$ is isomorphic to $\A(Fl \after \S_\C)$. 
The category of filters seems to have appeared first in \cite{KoRe1970}. More information on this construction may be found in \cite{Bla1977, Bu04}.

For a coherent category $\C$, its category of filters $\Lambda\C$ is again a coherent category. For a morphism $(A, F) \xrightarrow{[\alpha]} (B, G)$ in $\Lambda\C$, its image may be described as $(B, \exists_{[\alpha]}F)$, where
$$
\begin{array}{rcl}
\exists_{[\alpha]}F
&=&
\{V \in Sub_{\C}(B)\,|\,\alpha^*(V) \in F\}\\
&=&
\uparrow\!\{\exists_{\alpha} (W \wedge dom(\alpha))\,|\,W \in F\}.
\end{array}
$$
Every coherent category carries a natural Grothendieck topology which is defined as follows.

\begin{definition}
Let $\C$ be a coherent category and $A \in \C$. A sieve $S$ on $A$ is covering in the \emph{coherent topology} $J_{coh}$ on $\C$ iff there exists a finite subset $\{A_i \xrightarrow{\alpha_i} A \, | \, 1 \le i \le n\}$ of $S$ s.t.
$$
A = \bigvee \{\exists_{\alpha_i} A_i\,|\, 1 \le i \le n\},
$$ 
\it{i.e.}, $J_{coh}$ is the topology generated by finite joins and images.
\end{definition}

Makkai defines the topos of types as follows. 
\begin{definition}
Let $\C$ be a coherent category. The \emph{topos of types} $T(\C)$ of $\C$ is the topos of sheaves over $(\tau\C, J_p)$. Here $\tau\C$ is the full subcategory of the category of filters $\Lambda \C$ consisting of all pairs $(A, \rho)$, where $\rho$ is a prime filter in $Sub_\C(A)$, and $J_p$ is the topology induced by the coherent topology on $\Lambda\C$.
\end{definition}

For a coherent category $\C$, the topology $J_p$ on $\tau\C$ is the topology generated by the singleton covers. To see this, note that, in $\Lambda \C$, a sieve $S$ on $(A,F)$ is covering in the coherent topology if and only if there exists a finite subset $\{(A_i, F_i) \xrightarrow{[\alpha_i]} (A,F) \,|\, 1\le i \le n\}$ of $S$ such that, for all $U_1 \in F_1, \ldots, U_n \in F_n$, 
$$
\bigvee\{\exists_{\alpha_i}U_i\,|\,1 \le i \le n\} \in F.
$$
It follows that, in case $F$ is a prime filter in $Sub_\C(A)$, a sieve $S$ on $(A,F)$ is covering if and only if there is morphism $(A', F') \xrightarrow{[\alpha]} (A,F)$ in $S$ whose image is $(A,F)$. Hence, the topology on $\tau\C$, induced by the coherent topology on $\Lambda\C$, is the topology generated by the singleton covers.

For distributive lattices, the topos of types construction essentially yields their canonical extension.

\begin{proposition}
Let $\L$ be a distributive lattice. Viewing $\L$ as a coherent category, its topos of types $T(\L)$ is equivalent to the topos of sheaves over its canonical extension $\L^\delta$, viewed as a locale.
\end{proposition}

\begin{proof}
By definition, the category $\tau\L$ consists of all pairs $(a, \rho)$, where $a \in \L$ and $\rho$ is a prime filter in the downset of $a$. The category $\tau\L$ is a preorder and, for all objects $(a, \rho)$, $(a, \rho) \cong (1, \uparrow\!\rho)$, where $\uparrow\!\rho$ denotes the upset of $\rho$ in $\L$, which is a prime filter in $\L$. We write $\cat{E}_\L$ for the full subcategory of $\tau\L$ consisting of all pairs of the form $(1, \rho)$. By the Comparison Lemma, the topos of types $T(\L) = Sh(\tau_L, J_p)$ is equivalent to the topos of sheaves over $\cat{E}_\L$ with the topology induced by $J_p$. 

Note that $\cat{E}_\L$ is (isomorphic to) the poset $(PrFl(\L), \supseteq)$ of prime filters of $\L$ with the reverse inclusion order, as there exists a morphism $(1, \rho) \to (1, \rho')$ if and only if $\rho' \subseteq \rho$. Furthermore, the induced topology on $\cat{E}_\L$ is the trivial topology. Hence, the topos of types $T(\L) $ is equivalent to the topos of presheaves over $\cat{E}_\L \cong (PrFl(\L), \supseteq)$. As $\cat{E}_\L$ is a poset, this presheaf topos is localic. Its lattice of subterminal objects is isomorphic to the downset lattice of $(PrFl(\L), \supseteq)$, {\it i.e.}, to the canonical extension $\L^\delta$ of $\L$. Hence, the topos $T(\L)$ is equivalent to the topos of sheaves over the locale $\L^\delta$. \QED
\end{proof}

We will now give an alternative description of the topos of types $T(\C)$, of a coherent category $\C$, using the coherent hyperdoctrine $\S_\C^\delta$ defined in the previous section. For all $A \!\in\! \C$, $\S_\C^\delta(A)$ is a complete completely distributive lattice and therefore it is in particular a frame. Using that $\S_\C^\delta$ is a coherent hyperdoctrine, it follows from the description of internal locales in $\Set^{\C^{\op}}$ given in \cite{JoTi}, that $\S_\C^\delta$ is a locale in $\Set^{\C^{\op}}$. We will show that the topos of types of $\C$ is equivalent to the topos of sheaves over the internal locale $\S_\C^\delta$. The Comparison Lemma plays an essential role in the proof of this theorem. The basic form of this Lemma is to be found in \cite{SGA4}. We use a slightly more general form, as formulated in \cite{KoMo}.  

\begin{lemma}[Comparison Lemma, \cite{KoMo}]
Let $e \colon (\D, K) \to (\C, J)$ be a functor between essentially small sites satisfying
\begin{enumerate}
\item $e$ is \emph{cover preserving}, {\it i.e.}, for all $D \in \D$, if $S \in K(D)$, then the sieve $(e(S))$, generated by the image of $S$ in $\C$, is in $J(e(D))$;
\item $e$ is \emph{locally full}, {\it i.e.}, if $g \colon e(C) \to e(D)$ is a map in $\C$, then there exists a cover $(\xi_i \colon C_i \to C)_{i\in I}$ in $\D$ and maps $(f_i \colon C_i \to D)_{i\in I}$ such that, for all $i \in I$, $g \after e(\xi_i) = e(f_i)$;
\item $e$ is \emph{locally faithful}, {\it i.e.}, if $f, f'\colon C \to D$ in $\D$ with $e(f) = e(f')$, then there exists a cover $(\xi_i \colon C_i \to C)_{i\in I}$ such that, for all $i \in I$, $f \after \xi_i = f' \after \xi_i$;
\item $e$ is \emph{locally surjective on objects}, {\it i.e.}, for all $C \in \C$, there exists a covering family of the form $(e(C_i) \to C)_{i\in I}$;
\item $e$ is \emph{co-continuous}, {\it i.e.}, if $(\xi_i \colon C_i \to e(D))_{i \in I}$ is a cover in $\C$, then the set of arrows $f \colon D' \to D$ in $\D$, such that $e(f)$ factors through some $\xi_i$, covers $D$ in $\D$.
\end{enumerate}
Then the functor $e^* \colon Sh(\C, J) \to Sh(\D, K)$ given by $F \mapsto F \after e$, is an equivalence.
\end{lemma}

\begin{theorem}
\label{thm:inttt}
For a coherent category $\C$, the topos of types $T(\C)$ is equivalent to the topos of sheaves over the internal locale $\S_\C^\delta = (\,\,)^\delta \after Sub_\C$ in $\Set^{\C^{\op}}$.
\end{theorem}

\begin{proof}
Using a general construction in \cite{Mo}, we may describe an (external) site $(\C\!\ltimes\! \S_\C^\delta, J)$ such that the topos of sheaves over the internal locale $\S_\C^\delta$ is equivalent to the topos $Sh(\C\!\ltimes\! \S_\C^\delta, J)$ of sheaves over this site. We have to show that $Sh(\C\!\ltimes\! \S_\C^\delta, J)$ is equivalent to the topos of types $T(\C)$. 
The objects of $\cat{C} \!\ltimes\! \S_\C^\delta$ are pairs $(A, u)$, where $A \in \cat{C}$ and $u \in \S_\C^\delta(A)$. A morphism $(A, u) \to (B,v)$ is a morphism $A \xrightarrow{\alpha} B$ in $\cat{C}$ such that $u \le \S_\C^\delta(\alpha)(v)$. The Grothendieck topology $J$ on $\cat{C} \!\ltimes\! \S_\C^\delta$ is given by: a sieve $\{(A_i, u_i) \xrightarrow{\alpha_i} (A,u)\}_{i \in I}$ is a cover if and only if
$$
\bigvee \{\exists_{\alpha_i} u_i\,|\, i \in I\} = u,
$$
where $\exists_{\alpha_i}$ is the left adjoint of $\S_\C^\delta(\alpha_i)$. 
Let $\cat{D}$ be the full subcategory of $\cat{C} \!\ltimes\! \S_\C^\delta$ consisting of the objects of the form $(A, x)$, where $A \in \C$ and $x \in \mathcal{J}^\infty(S_\C^\delta(A))$. The induced topology $J'$ on $\D$ is the topology generated by the singleton covers. We will use the Comparison Lemma to prove
$$
Sh(\C \!\ltimes\! \S_\C^\delta, J) \simeq Sh(\D,J') \simeq Sh(\tau\C, J_p) = T(\C).
$$ 
We leave it to the reader to check that the inclusion $(\D, J') \hookrightarrow (\C \ltimes \S_\C^\delta, J)$ satisfies the conditions of the Comparison Lemma.
To define a functor $e \colon \D \to \tau\C$, recall from Section~\ref{sec:canextdl} that, for a distributive lattice $L$, the completely join-irreducible elements $\mathcal{J}^\infty(L^\delta)$ of its canonical extension $L^\delta$ correspond to the prime filters $PrFl(L)$ of $L$. We write
$$
\begin{array}{ll}
\text{for $\rho \in PrFl(L)$,} & x_\rho := \bigwedge \rho \in \mathcal{J}^\infty(L^\delta),\\
\text{for $x \in \mathcal{J}^\infty(L^\delta)$,} & \rho_x := \{a \in L\,|\,x \le a\} \in PrFl(L). 
\end{array} 
$$ 
As, for $A \in \C$, $\S_\C^\delta(A) = Sub_\C(A)^\delta$, the completely join-irreducible elements of $\S_\C^\delta(A)$ correspond to the prime filters in $Sub_\C(A)$. 
We define a functor $e \colon \D \to \tau\C$ by
$$
(A,x) \xrightarrow{\alpha} (B,z) \,\mapsto\, (A, \rho_x) \xrightarrow{[\alpha]} (B,\rho_z)
$$
We first show that this functor is well-defined, that is, for a morphism $(A,x) \xrightarrow{\alpha} (B, z)$ in $\D$, $\alpha$ is a (local) continuous map $(A, \rho_x) \to (B, \rho_z)$. Let $U \in \rho_z$, {\it i.e.}, $U \in Sub_\C(A)$ and $z \le U$ (in $\S_\C^\delta(B) = Sub_\C(B)^\delta$). We have to show $\alpha^*(U) \in \rho_x$. Recall that $\S_\C^\delta(\alpha) \colon \S_\C^\delta(B) \to \S_\C^\delta(A)$ is the canonical extension of $\alpha^* \colon Sub_\C(B) \to Sub_\C(A)$, hence $\S_\C^\delta(\alpha)(U) = \alpha^*(U)$. Using the fact that  $(A,x) \xrightarrow{\alpha} (B, z)$ is a morphism in $\C \!\ltimes \S_\C^\delta$ it follows that
$$
x \le \S_\C^\delta(\alpha)(z) \le \S_\C^\delta(\alpha)(U) = \alpha^*(U).
$$
Hence $\alpha^*(U) \in \rho_x$.

We now check that the functor $e \colon  \D \to \tau\C$ satisfies conditions 1 to 5 of the Comparison Lemma.
\begin{enumerate}
\item Let $S$ be a covering sieve of $(A,x)$ in $\D$. Then there exists $(B, z) \xrightarrow{\alpha} (A,x) \in S$ s.t. $\exists_{\alpha} z = x$. It follows that in $\tau\C$, 
$$
\begin{array}{rcll}
\exists_{[\alpha]} \rho_z &=& \{V \in Sub_\C(A)\,|\,\alpha^*(V) \in \rho_z\}\\
&=&
\{V \in Sub_\C(A)\,|\,z \le \alpha^*(V)\}\\
&=&
\{V \in Sub_\C(A)\,|\,\exists_\alpha z \le V\} &\text{(where $\exists_\alpha \dashv \S_\C^\delta(\alpha)$ and using $\alpha^*(V) = \S_\C^\delta(\alpha)(V)$)}\\
&=&
\{V \in Sub_\C(A)\,|\,x \le V\}\\
&=&
\rho_x.
\end{array}
$$
Whence $(e(S))$ covers $(A, \rho_x) = e(A,x)$.  
\item Let $e(A,x) = (A, \rho_x) \xrightarrow{[\alpha]} (B,\rho_z) = e(B,z)$ in $\tau\C$. Let $m \colon U \hookrightarrow A \in \rho_x$ such that $\alpha \colon U \to B$. Then $x$ is also completely join-irreducible in $\S_\C^\delta(U)$, where we view $\S_\C^\delta(U) = Sub_\C(U)^\delta$ as a subset of $\S_\C^\delta(A) = Sub_\C(A)^\delta$. The morphism $(U, x) \xrightarrow{m} (A,x)$ generates a covering sieve, as $m$ being mono implies that $\exists_m \colon \S_\C^\delta(U) \to \S_\C^\delta(A)$ is the inclusion map. Commutativity of the following diagram
$$
\xymatrix{
e(A,x) \ar[r]^-{[\alpha]} & e(B,z)\\
e(U,x) \ar[u]^-{e(m)} \ar[ur]_-{e(\alpha)}
}
$$
proves that this sieve satisfies the requirements.

Remark that in $\tau\C$ the maps $[id_U] \colon e(A, x) \leftrightarrows e(U, x) \colon [m]$ are each others inverse, whence, in $\tau\C$, $e(A,x) \cong e(U,x)$. However, in $\D$ only one of the two maps is present. 

\item Let $\alpha, \beta \colon (A, x) \to (B,z)$ in $\D$ such that $[\alpha] = [\beta]$. Then there exists, by definition, $m \colon U \hookrightarrow A \in \rho_x$ with $\alpha \circ m = \beta \circ m$. As above $\{(U,x) \xrightarrow{m} (A,x)\}$ generates a covering sieve. This sieve satisfies the requirements.

\item For $(A, \rho) \in \tau\C$, $e(A, x_\rho) = (A,\rho)$, so $e$ is surjective on objects.

\item Let $\{(A_i, \rho_i) \xrightarrow{[\alpha_i]} e(A, x)\}_{i \in I}$ be a covering sieve in $\tau\C$. For each $i$, let $U_i$ be the domain of $\alpha_i$. Then the arrow $e(U_i, x_{\rho_i}) \xrightarrow{e(\alpha_i)} e(A,x)$ factors through $(A_i, \rho_i) \xrightarrow{[\alpha_i]} e(A, x)$ and the family $\{(U_i, x_{\rho_i}) \xrightarrow{\alpha_i} (A,x)\}_{i \in I}$ generates a covering sieve in $\D$.
\end{enumerate}
Applying the Comparison Lemma twice, we have 
$$
Sh(\C \!\ltimes\! \S_\C^\delta, J) \simeq Sh(\D,J') \simeq Sh(\tau\C, J_p) = T(\C),
$$ 
which completes the proof.\QED
\end{proof}
\auxproof{
{\bf Claim 1:} The induced topology $J'$ on $\D$ is the topology generated by the singleton covers. \\
Let $S = \{(A_i, x_i) \xrightarrow{\alpha_i} (A, x)\}_{i \in I}$ be a sieve in $\D$. Then $S \in J'(A,x)$ if and only if the $(S) \in J(A,x)$, where $(S)$ is the sieve generated by $S$ in $\C \!\ltimes\! \S_\C^\delta$. Note that 
$$
\begin{array}{rcl}
(S) \in J(A,x) &\Leftrightarrow& \bigvee \{\exists_{\alpha_i} x_i\,|\,i\in I\} = x\\
&\Leftrightarrow&
\exists i \in I. \ \exists_{\alpha_i} x_i = x,
\end{array}
$$
where the last equivalence follows from the fact that $x$ is completely join-irreducible.\\
\\
{\bf Claim 2:} The inclusion $(\D, J') \to (\C \!\ltimes\! \S_\C^\delta, J)$ satisfies conditions 1 to 5 of the Comparison Lemma.
\begin{enumerate}
\item Immediate from the definition of $J'$.
\item $D$ is a full subcategory of $\C \!\ltimes\!S_\C^\delta$, so the inclusion functor is full.
\item The inclusion functor is faithful.
\item Let $(A,u) \in  (\C \!\ltimes\! \S_\C^\delta, J)$ and consider the familie $F = \{(A, x) \xrightarrow{\id_A} (A,u) \,|\,u \ge x \in J^\infty(S_\C^\delta(A))\}$. As $\S_\C^\delta(A)$ is join-generated by $J^\infty(S_\C^\delta(A))$, it follows that $\bigvee \{x \,|\, (A,x) \in F\} = u$. Hence, the sieve generated by $F$ covers $(A,u)$.
\item Let $(A,x) \in \D$ and let $\{(A_i, u_i) \xrightarrow{\alpha_i} (A,x)\}_{i \in I}$ be a covering sieve in $\S_\C^\delta(A)$. Then $\bigvee_I \exists_{\alpha_i} u_i = x$. As $x$ is completely join-irreducible, there exists $i \in I$ such that $\exists_{\alpha_i} u_i = x$. For $(A_i, z) \in \D$ with $z \le u_i$, $(A_i, z) \xrightarrow{\alpha_i} (A,x)$ is an arrow in $D$ which factors through $(A_i, u_i) \xrightarrow{\alpha_i} (A,x)$ in $\C \!\ltimes\! \S_\C^\delta$. Furthermore 
$$
\bigvee \{\exists_{\alpha_i} z\,|\, u_i \ge z \in \mathcal{J}^\infty(S_\C^\delta(A_i))\} = \exists_{\alpha_i} \bigvee \{ z\,|\, u_i \ge z \in \mathcal{J}^\infty(S_\C^\delta(A_i))\} = \exists_{\alpha_i} u_i = x.
$$
Again using join-irreducibility of $x$, there exists $z \in \mathcal{J}^\infty(S_\C^\delta(A_i))$ with $z \le u_i$ and $\exists_{\alpha_i} z = x$. From which the required property follows.
\end{enumerate}
} 

On a category $\D$ in $\Coh^+$, we may consider the topology generated by images and \emph{arbitrary} joins, {\it i.e.}, the topology where a sieve $S$ on $A \in \D$ is covering iff $\bigvee \{\exists_{\alpha} B \,|\, B \xrightarrow{\alpha} A \in S\} = A$. We denote this topology by $J_{coh^+}$. For a coherent category $\C$, $\tau\C$ is (isomorphic to) the full subcategory of $\C^\delta = \A(\S_\C^\delta)$ consisting of all pairs $(A, x)$, with $A \in \C$ and $x \in \mathcal{J}^{\infty}(\S_\C^\delta(A))$. The topology $J_p$ on $\tau\C$ is the topology induced by the topology $J_{coh^+}$ on $\C^\delta$. Using this, it readily follows that also $(\C^\delta, J_{coh^+})$ is a site for the topos of types $T(\C)$.

\subsection{Involving the coherent topology}
When working in $\Set^{\C^{\op}}$, for a coherent category $\C$, one only remembers the limit structure of $\C$, {\it i.e.}, the Yoneda embedding $y \colon \C \to \Set^{\C^{\op}}$ only preserves the finite limits of $\C$ and not the joins and images. Therefore, it is more natural to consider $\C$ with the coherent topology $J_{coh}$ and to work in $Sh(\C, J_{coh})$, the classifying topos of $\C$. In this section we show that, for a coherent category $\C$, the functor $\S_\C^\delta$ is a sheaf over $(\C, J_{coh})$. The topos of internal sheaves over $\S_\C^\delta$ in $Sh(\C, J_{coh})$ is equivalent to the topos of internal sheaves over $\S_\C^\delta$ in $\Set^{\C^{\op}}$, as we prove in Theorem~\ref{thm:shjcoh}.

To prove that $\S_\C^\delta$ is a sheaf over $(\C, J_{coh})$ we start with a lemma.

\begin{lemma}
\label{lem:uniqgl}
Let $\C$ be a coherent category and let $\{A_i \xrightarrow{\alpha} A \,|\,i \in I\}$ be a finite collection of morphisms in $\C$ s.t. $\bigvee \exists_{\alpha_i}A_i = A$. For all $u \in \S_\C^{\delta}(A)$, $u = \bigvee \exists_{\alpha_i}^\delta(\S_\C^\delta(\alpha_i)(u))$.
\end{lemma}

\begin{proof}
As usual, we view $\S_\C(A_i)$ as a subset of $\S_\C^\delta(A_i)$. The top element of $\S_\C(A_i)$ is the top element of $\S_\C^\delta(A_i)$, {\it i.e.}, $1_{\S_\C^\delta(A_i)} = A_i \hookrightarrow A_i \in Sub_\C(A_i) \subseteq \S_\C^\delta(A_i)$. For $u \in \S_\C^{\delta}(A)$,
$$
\begin{array}{rcll}
\bigvee \{\exists_{\alpha_i}^\delta(\S_\C^\delta(\alpha_i)(u))\,|\,i \in I\}
&=&
\bigvee \{\exists_{\alpha_i}^\delta(\S_\C^\delta(\alpha_i)(u) \wedge 1_{\S_\C^\delta(A_i)})\,|\,i \in I\}\\
&=&
\bigvee \{u \wedge \exists_{\alpha_i}^\delta(1_{\S_\C^\delta(A_i)})\,|\, i \in I\} &\text{(Frobenius)}\\
&=&
u \wedge \bigvee \{\exists_{\alpha_i}(A_i)\,|\, i \in I\}&\text{(distributivity)}\\
&=&
u \wedge A = u,
\end{array}
$$
as required.\QED
\end{proof}

To ease the notation, in the remainder of this section we write, for $A \xrightarrow{\alpha} B$ in $\C$, $\exists_\alpha$ both for the left adjoint to the pullback functor $\alpha^* \colon Sub_\C(B) \to Sub_\C(A)$ and for its canonical extension, which is the left adjoint to $\S_C^\delta(\alpha)$. The intended meaning should be clear from the context.

\begin{proposition}
The functor $\S_\C^\delta \colon \C^{\op} \to \Sets$ is a sheaf over $(\C, J_{coh})$.
\end{proposition}

\begin{proof}
Let $\{A_i \xrightarrow{\alpha} A \,|\,i \in I\}$ be a finite collection of morphisms in $\C$ s.t. $\bigvee \exists_{\alpha_i}A_i = A$ and let $\{u_i \in \S_\C^\delta(A_i)\}_{i \in I}$ be a matching family. We have to show that there exists a unique element $u \in \S_\C^\delta(A)$ with $\S_\C^\delta(\alpha_i)(u) = u_i$, for all $i \in I$. Consider $u = \bigvee \{\exists_{\alpha_j} u_j \,|\,j \in I\}$. For $i \in I$,
$$
\S_\C^\delta(\alpha_i)(u) \,=\, \S_\C^\delta(\alpha_i)(\bigvee \{\exists_{\alpha_j} u_j \,|\,j \in I\}) \,\ge\, \S_\C^\delta(\alpha_i)(\exists_{\alpha_i} u_i) \,\ge\, u_i.
$$
To prove $\S_\C^\delta(\alpha_i)(u) \le u_i$, first recall that $\S_\C^\delta(\alpha)$ preserves all joins and therefore
$$
\S_\C^\delta(\alpha_i)(u) = \bigvee \{\S_\C^\delta(\alpha_i)(\exists_{\alpha_j} u_j) \,|\,j \in I\}
$$
We have to show that, for all $j \in I$, $\S_\C^\delta(\alpha_i)(\exists_{\alpha_j} u_j) \le u_i$. Let $j \in J$ and consider the pullback diagram
$$
\xymatrix{
A_i \times_A A_j \ar[r]^-{\gamma_i} \ar[d]_-{\gamma_j} & A_i \ar[d]^-{\alpha_i}\\
A_j \ar[r]_-{\alpha_j} & A
}
$$
Using the fact that $\S_\C^\delta$ is a coherent hyperdoctrine over $\C$, 
$$
\begin{array}{rcll}
\S_\C^\delta(\alpha_i)(\exists_{\alpha_j} u_j)
&=&
\exists_{\gamma_i}(\S_\C^\delta(\gamma_j)(u_j)) &\text{(Beck-Chevalley)}\\
&=&
\exists_{\gamma_i}(\S_\C^\delta(\gamma_i)(u_i)) &\text{(matching condition)}\\
&\le&
u_i &\text{(adjunction property).}
\end{array}
$$
Uniqueness of $u$ follows from Lemma~\ref{lem:uniqgl}.\QED
\end{proof}

We conclude this section by showing that the topos of internal sheaves over $\S_\C^\delta$ in $Sh(\C, J_{coh})$ is equivalent to the topos of internal sheaves over $\S_\C^\delta$ in $\Set^{\C^{\op}}$, and therefore, by Theorem~\ref{thm:inttt}, to the topos of types $T(\C)$. We write $\widehat{\C} = \Set^{\C^{\op}}$ and $\widetilde{\C} = Sh(\C, J_{coh})$.

\begin{theorem}
\label{thm:shjcoh}
Let $\C$ be a coherent category. Then $Sh_{\widehat{\C}}(\S_\C^\delta) \simeq Sh_{\widetilde{\C}}(\S_\C^\delta)$.
\end{theorem}

\begin{proof}
As in the proof of Theorem~\ref{thm:inttt}, we use the construction of \cite{Mo} to describe the topoi $Sh_{\widehat{\C}}(\S_\C^\delta)$ and $Sh_{\widetilde{\C}}(\S_\C^\delta)$ as sheaves over an external site. Both external sites have the same underlying category $\C\!\ltimes\!\S_\C^\delta$. We will show that also the two topologies coincide. We write $J_{\widehat{\C}}$ (resp. $J_{\widetilde{\C}}$) for the topology on $\C\!\ltimes\!\S_\C^\delta$ corresponding to $Sh_{\widehat{\C}}(\S_\C^\delta)$ (resp. $Sh_{\widetilde{\C}}(\S_\C^\delta)$). Let $S$ be a sieve on $(A,u)$. As described in the proof of Theorem~\ref{thm:inttt}, $S$ is a cover in $J_{\widehat{\C}}$ if and only if
$$
\bigvee \{\exists_{\beta} v\,|\, (B, v) \xrightarrow{\beta} (A,u) \in S\} = u.
$$ 
To ease the notation we write $(\beta, v) \in S$ for $(B, v) \xrightarrow{\beta} (A,u) \in S$. When describing the topology $J_{\widetilde{\C}}$, we have to take the coherent topology on $\C$ into account. The sieve $S$ is a cover in $J_{\widetilde{\C}}$ if and only if
$$
\bigvee \{\exists_{\gamma} w\,|\, \gamma \colon C \to A, w \in \S_\C^\delta(C), (\gamma, w) \in \overline{S}\} = u,
$$ 
where $(\gamma, w) \in \overline{S}$ if and only if there exists a cover $\{C_k \xrightarrow{\gamma_k} C\}_{k\in K}$ in the coherent topology on $\C$ (where $C$ is the domain of $\gamma$) such that, for all $k \in K$, $(\gamma \after \gamma_k, \S_\C^\delta(\gamma_k)(w)) \in S$. 
We will show 
$$
\bigvee \{\exists_{\beta} v\,|\, (\beta, v) \in S\} = \bigvee \{\exists_{\gamma} w\,|\, (\gamma, w) \in \overline{S}\} ,
$$ 
which implies that the two topologies coincide.
As $(\beta, v) \in S$ implies $(\beta, v) \in \overline{S}$, clearly $\bigvee \{\exists_{\beta} v\,|\, (\beta, v) \in S\} \le \bigvee \{\exists_{\gamma} w\,|\, (\gamma, w) \in \overline{S}\}$. To prove the converse inequality, let $C \xrightarrow{\gamma} A$ and $w \in \S_\C^\delta(C)$ such that $(\gamma, w) \in \overline{S}$. Then there exists a finite collection $\{C_k \xrightarrow{\gamma_k} C \,|\, 1 \le k \le n\}$ of morphisms in $\C$ such that $\bigvee \{\exists_{\gamma_k} C_k \,|\,1 \le k\le n\} = C$ and $(\gamma \after \gamma_k, \S_\C^\delta(\gamma_k)(w)) \in S$, for all $k$. Using Lemma~\ref{lem:uniqgl},
$$
\begin{array}{rcll}
\exists_\gamma(w) 
&=&
\exists_\gamma(\bigvee\{\exists_{\gamma_k}(\S_\C^\delta(\gamma_k)(w))\,|\,1\le k\le n\})\\
&=&
\bigvee\{\exists_\gamma(\exists_{\gamma_k}(\S_\C^\delta(\gamma_k)(w)))\,|\,1\le k\le n\}\\
&=& 
\bigvee\{\exists_{\gamma\after\gamma_k}(\S_\C^\delta(\gamma_k)(w))\,|\,1\le k\le n\}\\
&\le&
\bigvee \{\exists_{\beta} v\,|\, (\beta, v) \in S\},
\end{array}
$$    
which proves the claim.\QED
\end{proof}

\section{Properties of the topos of types}
\label{sec:proptt}
In this final section we apply our alternative description of the topos of types construction to investigate some of its properties. In Section~\ref{sec:propttmor} we study the action of the topos of types construction on morphisms. Of central importance in this section is the method of proof of the main theorem (Theorem~\ref{thm:mortt}), as it enlightens the relationship between properties of the topos of types construction for coherent categories and properties of the canonical extension construction for distributive lattices.
In Section~\ref{sec:ttmod} we describe, for a coherent category $\C$, a relationship between its topos of types $T(\C)$ and the class of models of $\C$ (in $\Set$). This result (Theorem~\ref{thm:hcfev}) is new, as far as we know. 

\subsection{The topos of types construction on morphisms}
\label{sec:propttmor}
In this section we study the action of the topos of types construction on morphisms. The main theorem of this section is the following.

\begin{theorem}
\label{thm:mortt}
Let $F \colon \C \to \D$ be a coherent functor.
\begin{enumerate}
\renewcommand{\theenumi}{\roman{enumi}}
\item if $F$ is conservative, then $T(F) \colon T(\D) \to T(\C)$ is a geometric surjection;
\item if $F$ is a morphism of Heyting categories, then $T(F)$ is an open geometric morphism. 
\end{enumerate}
\end{theorem}

Item (i) appears to be new, item (ii) was proved already in \cite{MakRey95}, Corollary 8.6. However, our method of proof is very different from the approach in \cite{MakRey95}. Using our representation of the topos of types as sheaves over the internal locale $\S_\C^\delta$, we may rely on results on the canonical extension construction for distributive lattices and the relationship between properties of internal locale morphisms and properties of the correponding geometric morphisms. This allows a transparent proof of the above theorem which exposes the analogue with the algebraic situation.

A geometric morphism is surjective (resp. open) if and only if its localic part is surjective (resp. open). For a coherent functor $F \colon \C \to \D$, both $\S_\C^{\delta}$ and $\S_\D^{\delta} \after F$ are locales in $\Set^{\C^{\op}}$ and $F$ induces a morphism of locales $\tau^F \colon \S_\D^{\delta} \after F \to \S_\C^\delta$, where, for $A \in \C$, the frame homomorphism $(\tau_A^F)^* \colon \S_\C^{\delta}(A) \to \S_\D^{\delta}(FA)$ is the canonical extension of the restriction of $F$ to a map $F_A \colon \S_\C(A) \to \S_\D(FA)$. We will also write $F_A^\delta$ for $(\tau_A^F)^*$. 
This internal morphism of locales in turn gives rise to a geometric morphism $\psi \colon Sh(\S_\D^{\delta} \after F) \to Sh(\S_\C^{\delta}) \simeq T(\C)$, between the topoi of sheaves over the internal locales $\S_\D^{\delta} \after F$ and $\S_\C^\delta$, respectively. We will show that $\psi$ is the localic part of $T(F) \colon T(\D) \to T(\C)$ and then prove Theorem~\ref{thm:mortt} by studying this localic part.  We rely on the fact that, for a locale $X$ in $\Set^{\C^{\op}}$, the subobject classifier in $Sh_{\widehat{C}}(X)$ may be described as,
$$
\begin{array}{rcl}
\Omega \colon (\C \!\ltimes\! X)^{\op} &\to &\Sets \\
\xymatrix{(A, u) \ar[d]_-{\alpha}\\ (B, v)} &\mapsto& \xymatrix{\downarrow_{\scriptscriptstyle{X(A)}}\! u \\ \downarrow_{\scriptscriptstyle{X(B)}}\! v \ar[u]_-{w \mapsto X(\alpha)(w) \wedge u}}
\end{array}
$$

\auxproof{
Proof of the fact that $\tau^F$ is an internal frame homomorphism.
As all components of $\tau^F$ are complete lattice homomorphisms, they are in particular morphisms of frames. Using furthermore that $(F, (\tau^F)^\delta) \colon \S_\C^\delta \to \S_\D^\delta$ is a morphism of coherent hyperdoctrines, whence preserves existential quantification, it follows tat $\tau^F$ is an internal frame homomorphism.
}

\auxproof{
We show that, for $(A,u) \in \C\!\ltimes\! X$, the closed sieves on $(A, u)$ correspond to the elements of $\downarrow_{\scriptscriptstyle{X(A)}}\! u$. Define a mapping
$$
\begin{array}{rcll}
h \colon \{S\,|\,\text{$S$ is a closed sieve on $(A,u)$}\} &\to& \downarrow_{\scriptscriptstyle{X(A)}} u\\
S &\mapsto& \bigvee\{\exists_{\alpha}v\,|\,(B, v) \xrightarrow{\alpha} (A, u) \in S\}.
\end{array}
$$
We first show that, for any closed sieve $S$ on $(A,u)$ and $(B, v) \xrightarrow{\alpha} (A, u)$ in $\C \!\ltimes\! X$,
\begin{equation}
\label{eq:injh}
\alpha \in S \quad\Leftrightarrow\quad \exists_{\alpha} v \le h(S),
\end{equation}
where $\exists_{\alpha}$ is the left adjoint of $X(\alpha)$. Clearly the left to right implication holds, by definition of $h$. For the converse implication note that $\{(B, v) \xrightarrow{\alpha} (A, h(S))\,|\, (B, v) \xrightarrow{\alpha} (A, u) \in S\}$ is a cover in $\C \!\ltimes\! X$. Hence, as $S$ is closed, $(A, h(S)) \xrightarrow{\id} (A, u) \in S$. A morphism $(B, v) \xrightarrow{\alpha} (A, u)$ with $\exists_{\alpha} v \le h(S)$, may be factored as
$$
(B, v) \xrightarrow{\alpha} (A, h(S)) \xrightarrow{\id} (A, u).
$$
As $S$ is a sieve, then $(B, v) \xrightarrow{\alpha} (A, u) \in S$.

It follows from \eqref{eq:injh} that $h$ is injective. To prove surjectivity, let $z \in \downarrow_{\scriptscriptstyle{X(A)}} u$ and define
$$
S_z = \{(B,v) \xrightarrow{\alpha} (A, u)\,|\, \exists_{\alpha}v \le z\}.
$$
We will show that $S_z$ is a closed sieve on $(A,u)$. For $(B,v) \xrightarrow{\alpha} (A, u) \in S_z$ and $(C, w) \xrightarrow{\beta} (B, v)$ in $\C\!\ltimes\! X$,
$$
\begin{array}{rcll}
\exists_{\alpha \after \beta} (w)
&=&
\exists_{\alpha}(\exists_{\beta}(w))\\
&\le&
\exists_{\alpha}(v) &(\text{definition morphisms in $\C\!\ltimes\! X$})\\
&\le&
z&((B,v) \xrightarrow{\alpha} (A, u) \in S_z).
\end{array}
$$
So $S_z$ is a sieve. Let $(B, v) \xrightarrow{\alpha} (A, u)$ and $\{(B_i, v_i) \xrightarrow{\alpha_i} (B, v)\}_{i \in I}$ a cover such that, for each $i \in I$, $(B_i, v_i) \xrightarrow{\alpha \after \alpha_i} (A, u) \in S_z$. Then
$$
\begin{array}{rcll}
\exists_{\alpha}(v)
&=&
\exists_{\alpha}(\bigvee_I \exists_{\alpha_i}v_i)\\
&=&
\bigvee_I \exists_{\alpha}(\exists_{\alpha_i}(v_i))\\
&=&
\bigvee_I \exists_{\alpha \after \alpha_i}v_i\\
&\le&
z.
\end{array}
$$
So $(B,v) \xrightarrow{\alpha} (A, u) \in S_z$ and $S_z$ is a closed sieve on $(A,u)$.

Clearly $h$ preserves and reflects the order, hence it is an order isomorphism. For a morphism $(B,v) \xrightarrow{\alpha} (A, u)$, pulling back a closed sieve on $(A,u)$ to a sieve on $(B,v)$ corresponds under the isomorphims $h$ to the mapping
$$
\begin{array}{rcll}
\downarrow_{\scriptscriptstyle{X(A)}} u &\to& \downarrow_{\scriptscriptstyle{X(B)}} v\\
w &\mapsto& X(\alpha)(w) \wedge v,
\end{array}
$$  
which completes the proof.
}

\begin{proposition}
\label{prop:factortt}
Let $F \colon \C \to \D$ be a morphism of coherent categories. The hyperconnected localic factorisation of $T(F) \colon T(\D) = Sh_{\widehat{\D}}(\S_\D^{\delta}) \to Sh_{\widehat{\C}}(\S_\C^{\delta}) = T(\C)$ is given by
$$
\xymatrix{
&Sh_{\widehat{\C}}(\S_\D^\delta \after F) \ar[rd]^-{\psi}\\
T(\D) = Sh_{\widehat{\D}}(\S_\D^{\delta}) \ar[rr]_-{T(F)} \ar[ru]^-{\phi} && Sh_{\widehat{\C}}(\S_\C^{\delta}) = T(\C)
}
$$
where $\phi$ is the geometric morphism induced by the morphism of sites
$$
\begin{array}{rcl}
\C \ltimes (\S_\D^\delta \after F) &\to& \D \ltimes \S_\D^\delta\\
(A,w) &\mapsto& (FA, w)
\end{array}
$$
and $\psi$ is the geometric morphism induced by the internal morphism of locales $\tau^F \colon \S_\D^\delta \after F \to \S_\C^\delta$, as described above Definition~\ref{def:cons}.
\end{proposition}

\begin{proof}
As $\psi$ is induced by an internal morphism of locales, it is localic. It remains to show that $\phi$ is hyperconnected. We write $\mathcal{E} = Sh_{\widehat{\C}}(\S_\D^\delta \after F)$. The geometric morphism $\phi$ is hyperconnected if and only if the comparison map $\phi_{*}(\Omega_{T(\D)}) \to \Omega_{\mathcal{E}}$ is an isomorphism. For $(A, w) \in \C \ltimes (\S_\D^\delta \after F)$,
$$
\begin{array}{rcl}
\phi_{*}(\Omega_{T(\D)}) 
&\cong&
\Omega_{T(\D)}(FA, w)\\
&\cong&
\downarrow_{\S_\D^\delta(FA)} w\\
&\cong&
\Omega_{\mathcal{E}}(A,w),
\end{array}
$$ 
which proves the claim. \QED
\end{proof}

In the next proposition we state some well-known facts about the relation between properties of internal locale morphisms and properties of the corresponding geometric morphisms. 

\begin{proposition}
\label{prop:loctotop}
Let $\tau \colon Y \to X$ be an internal locale morphism in $\Set^{\C^{\op}}$ and $f \colon Sh(Y) \to Sh(X)$ the induced geometric morphism between the topoi of internal sheaves over $Y$ and $X$. If $\tau$ is a surjection of locales, then $f$ is a geometric surjection. If $\tau$ is an open map of locales, then $f$ is an open geometric morphism. 
\end{proposition}

\auxproof{
First assume $\tau$ is a surjection of locales, {\it i.e.}, for all $A \in \C$, $\tau_A^* \colon Y(A) \to X(A)$ is an embedding. The geometric morphism $f \colon Sh(Y) \to Sh(X)$ is a surjection if and only if the unique map $\lambda \colon \Omega_{Sh(X)} \to f_*(\Omega_{Sh(Y)})$ is a monomorphism in $Sh(X)$. We use the representation of $Sh(X)$ (resp. $Sh(Y)$) as sheaves over the external site $\C\!\ltimes\! X$ (resp. $\C\!\ltimes\! Y$). Using Lemma~\ref{lem:subobj}, for $(A, u) \in \C\!\ltimes\! X$,
$$
\begin{array}{rcl}
f_*(\Omega_{Sh(Y)})(A, u) &=& \Omega_{Sh(Y)}(A, \tau_A(u))\\
&=& \downarrow_{\scriptscriptstyle{Y(A)}}\! \tau_A(u). 
\end{array}
$$
The component of $\lambda$ at $(A, u) \in \C \!\ltimes\! X$ is the restriction of $\tau_A$ to a map
$$
\lambda_{(A,u)} \colon \,\Omega_{Sh(X)}(A,u) = \downarrow_{\scriptscriptstyle{X(A)}}\! u \,\to\, \downarrow_{\scriptscriptstyle{Y(A)}}\! \tau_A(u) = f_*(\Omega_{Sh(Y)})(A, u).
$$
As all components of $\tau$ are embeddings, also the components of $\lambda$ are embeddings. Hence, $\lambda$ is a monomorphism in $Sh(\C \!\ltimes\! X) \simeq Sh(X)$.

Now assume $\tau$ is an open map of locales, {\it i.e.}, $\tau^*$ has an internal left adjoint $\sigma$ satisfying Frobenius. The geometric morphism $f$ is open if and only if the unique map $\lambda \colon \Omega_{Sh(X)} \to f_*(\Omega_{Sh(Y)})$ has an internal left adjoint, that is, if and only if the components $\lambda_{(A, u)}$ have a left adjoint, natural in $(A,u)$.
As, for $A \in \C$, $\sigma_A$ is left adjoint to $\tau_A$, it restricts to a map $\downarrow_{\scriptscriptstyle{Y(A)}}\! \sigma_A(u) \to \downarrow_{\scriptscriptstyle{X(A)}}\! u$ and this restriction is clearly left adjoint to the restriction of $\tau_A$, {\it i.e.}, to $\lambda_{(A,u)}$. To prove naturality in $(A,u)$ of these left adjoints, let $(B, v) \xrightarrow{\alpha} (A, u)$ in $\C \!\ltimes\! X$ and $w \in \downarrow_{\scriptscriptstyle{Y(A)}} \tau_A(u)$. Then
$$
\begin{array}{rcll}
\Omega_{Sh(X)}(\alpha)(\sigma_A(w))
&=&
X(\alpha)(\sigma_A(w)) \wedge v\\
&=&
\sigma_B(Y(\alpha)(w)) \wedge v &\text{(naturality $\sigma$)}\\
&=&
\sigma_B(Y(\alpha)(w) \wedge \tau_B(v)) &\text{(Frobenius)}\\
&=&
\sigma_B(f_*(\Omega_{Sh(Y)})(\alpha)),
\end{array}
$$
as required.
}

It follows that, to prove Theorem~\ref{thm:mortt}, it suffices to study the internal morphism of locales $\tau^F \colon \S_\D^{\delta} \after F \to \S_\C^{\delta}$.

\begin{definition}
\label{def:cons}
A coherent functor $F \colon \C \to \D$ is \emph{conservative} if and only if, for all $A \in \C$, for all $U, V \in Sub_\C(A)$, $FU \le FV$ in $Sub_\D(FA)$ implies $U \le V$ in $Sub_\C(A)$, {\it i.e.}, for all $A \in \C$, the restriction $F_A \colon Sub_\C(A) \to Sub_\D(FA)$ is an order-embedding. 
\end{definition}

\begin{proposition}
\label{lem:cohtosurj}
Let $F \colon \C \to \D$ be a coherent functor. If $F$ is conservative, then the natural transformation $\tau^F \colon \S_\D^{\delta} \after F \to \S_\C^{\delta}$ is a surjection of locales. 
\end{proposition}

\begin{proof}
The natural transformation $\tau^F \colon \S_\D^{\delta} \after F \to \S_\C^{\delta}$ is a surjection of locales if and only if all the maps $\tau_A^F = F_A^\delta \colon \S_\C^{\delta(A)} \to \S_\D^\delta(FA)$ are order-embeddings. This is indeed the case as, for each $A$ in $\C$, $F_A \colon \S_\C(A) \to \S_\D(FA)$ is an order-embedding and this property is preserved under canonical extension.\QED
\end{proof}

To prove a similar relationship between Heyting functors and open maps of locales, we use the following property of adjunctions between partially ordered sets. 

\begin{lemma}
\label{lem:comadj}
Consider the following pairs of maps between partially ordered sets.
$$
\xymatrix{
A \ar@<1ex>[r]^-{f} & B \ar@<1ex>[l]^-{f'} \ar@<1ex>[r]^-{g} & C  \ar@<1ex>[l]^-{g'}
}
$$
If $f'$ is left adjoint to $f$ and $g'$ is left adjoint to $g$, then $f'\after g'$ is left adjoint to $g \after f$.
\end{lemma}

\auxproof{
\begin{array}{rcl}
f'(g'(c)) \le a 
&\Leftrightarrow&
g'(c) \le f(a)\\
&\Leftrightarrow&
c \le g(f(a)).
\end{array}
}

\begin{proposition}
\label{lem:Heytopen}
Let $F \colon \C \to \D$ be a Heyting functor. Then $\tau^F \colon \S_\D^\delta \after F \to \S_\C^\delta$ is an open map of locales. 
\end{proposition}

\begin{proof}
The natural transformation $\tau^F$ is an open map of locales if and only if the internal frame morphism $(\tau^F)^*$ has an internal left adjoint in $\Set^{\C^{\op}}$ satisfying Frobenius. As the internal order in $\S_\C^\delta$ and $\S_\D^\delta \after F$ is computed component-wise, this left adjoint, if it exists, is given by taking component-wise left adjoints. For all $A \in \C$, $(\tau_A^F)^* = F_A^{\delta} \colon \S_\C^\delta(A) \to \S_\D^\delta(FA)$ is a complete lattice homomorphism and therefore it has a left adjoint $\sigma_A$. 

We have to prove that these maps $\sigma_A$ constitute a natural transformation, {\it i.e.}, a morphism in $\Set^{\C^{\op}}$. In order to do this, we rely on Lemma~\ref{lem:comadj} and use the fact that the morphisms involved in the naturality diagram all have right adjoints which interact appropriately. Let $A \xrightarrow{\alpha} B$ in $\C$. We have to show that the inner square of the following diagram commutes. 
$$
\xymatrix@R=3.5pc@C=3.5pc{
\S_\D^\delta(FB) \ar@<-1ex>[r]_-{\sigma_B}^-{{}_\top} \ar@<1ex>[d]^-{((F\alpha)^*)^\delta}_-{\,\,\,\scriptstyle{\vdash}} & \S_\C^\delta(B) \ar@<-1ex>[l]_-{F_B^\delta} \ar@<-1ex>[d]_-{(\alpha^*)^\delta}\\
\S_\D^\delta(FA) \ar@<1ex>[r]^-{\sigma_A}_-{{}^\bot} \ar@<1ex>[u]^-{\forall_{F\alpha}^\delta} & \S_\C^\delta(A) \ar@<1ex>[l]^-{F_A^{\delta}} \ar@<-1ex>[u]_-{\forall_{\alpha}^\delta}^-{\scriptstyle{\dashv}} 
}
$$
As $F$ is a Heyting functor, it preserves universal quantification and this property is preserved under canonical extension. This implies that the outer square commutes. Using Lemma~\ref{lem:comadj} and the uniqueness of adjoints, it follows that also the inner diagram commutes. This proves naturality of $\sigma$.

For the Frobenius condition, let $A \in \C$, $v \in \S_\C^{\delta}(A)$ and $w \in \S_\D^{\delta}(FA)$. We have to show
$$
\sigma_A(w \wedge F_A^\delta(v)) = \sigma_A(w) \wedge v.
$$
This is equivalent to the condition that, for all $v \in \S_\C^{\delta}(A)$, the inner square of the following diagram commutes.
$$
\xymatrix@R=3.5pc@C=3.5pc{
\S_\D^\delta(FA) \ar@<-1ex>[r]_-{\sigma_A}^-{{}_\top} \ar@<1ex>[d]^-{\_\, \wedge F_A^\delta(v)}_-{\,\,\,\scriptstyle{\vdash}} & \S_\C^\delta(A) \ar@<-1ex>[l]_-{F_A^\delta} \ar@<-1ex>[d]_-{\_\, \wedge v}\\
\S_\D^\delta(FA) \ar@<1ex>[r]^-{\sigma_A}_-{{}^\bot} \ar@<1ex>[u]^-{F_A^\delta(v) \to \,\_} & \S_\C^\delta(A) \ar@<1ex>[l]^-{F_A^{\delta}} \ar@<-1ex>[u]_-{v \to \,\_}^-{\scriptstyle{\dashv}} 
}
$$
As $F$ is Heyting functor, $F_A$ preserves implication and this property is preserved under canonical extension. Hence, the outer square commutes. As above, it follows that the inner square commutes as well then. See also Proposition V.1 in \cite{JoTi}.  \QED
\end{proof}

The above results add up to a proof of Theorem~\ref{thm:mortt}, formulated at the beginning of this section.

\begin{proof}[Proof of Theorem~\ref{thm:mortt}] Let $F \colon \C \to \D$ be a coherent functor. Suppose $F$ is conservative. Then, by Proposition~\ref{lem:cohtosurj}, $\tau^F \colon \S_\D^{\delta} \after F \to \S_\C^{\delta}$ is a internal surjection of locales. Hence, by Proposition~\ref{prop:loctotop}, the induced geometric morphism $Sh(\S_\D^{\delta} \after F) \to \S_\C^{\delta}$ is a geometric surjection. Finally, it follows from Proposition~\ref{prop:factortt} that the localic part of $T(F)$, and therefore $T(F)$ itself, is a geometric surjection. The second claim follows similarly from the mentioned propositions and Proposition~\ref{lem:Heytopen}.\QED
\end{proof}

\subsection{The topos of types and models of $\C$}
\label{sec:ttmod}
An internal locale $X$ in a topos $\mathcal{E}$ induces a localic geometric morphism $Sh_{\mathcal{E}}(X) \to \mathcal{E}$, where $Sh_{\mathcal{E}}(X)$ is the topos of internal sheaves over $X$. In particular, for a coherent category $\C$, the internal locale $\S_\C^\delta$ in $Sh(\C, J_{coh}) = \widetilde{\C}$ induces a localic geometric morphism
\begin{equation}
\label{eq:phit}
\phi_t \colon T(\C) = Sh_{\widetilde{\C}}(\S_\C^\delta) \to Sh(\C, J_{coh}).
\end{equation}
The new main result of this section is Theorem~\ref{thm:hcfev}. Here we prove that the geometric morphism $\phi_t$ is the localic part of a geometric morphism which naturally arises from the class of models of $\C$ (in $\Set$).\footnote{I thank Steve Awodey for suggesting this.} 

Let $\C$ be a coherent category. We write $Mod(\C)$ for the category of coherent functors $\C \to \Set$ (models of $\C$) with natural transformations. Now let $\Sp$ be a small full subcategory of $Mod(\C)$. There is a natural {\it evaluation functor} $ev \colon \C \to \Set^\Sp$ which sends an object $A \in \C$ to the functor
$$
\begin{array}{rcl}
ev(A) \colon \Sp &\to& \Set\\
M &\mapsto & M(A)\\
M \xrightarrow{\sigma} N &\mapsto& M(A) \xrightarrow{\sigma_A} N(A).
\end{array}
$$
A morphism $A \xrightarrow{\alpha} B$ in $\C$ is sent to the natural transformation $ev(A) \xrightarrow{ev(\alpha)}  ev(B)$, whose component at $M \in Mod(\C)$ is $M(A) \xrightarrow{M(\alpha)} M(B)$.
The following proposition is due to Joyal. A proof may be found in \cite{MakRey77}, see 6.3.5.
\auxproof{
One has to check that the functor $ev$ is well-defined. As an example we show that, for $A \xrightarrow{\alpha} B$, $ev(\alpha)$ is a natural transformation and we leave the other details to the reader. Let $M \xrightarrow{\sigma} N$ be a morphism in $Mod(\C)$ and consider
$$
\xymatrix@C=4pc{
ev(A)(M) \ar[d]_-{ev(\alpha)_M} \ar[r]^-{ev(A)(\sigma)} & ev(A)(N) \ar[d]^-{ev(\alpha)_N}\\
ev(B)(M) \ar[r]_-{ev(B)(\sigma)} & ev(B)(N)
}
$$
Filling out the definitions and using the naturality of $\sigma$ yields
$$
\begin{array}{rcl}
ev(\alpha)_N \after ev(A)(\sigma) &=& N(\alpha) \after \sigma_A\\
&=&
\sigma_B \after M(\alpha)\\
&=&
ev(B)(\sigma) \after ev(\alpha)_M.
\end{array}
$$
}

\begin{proposition}
The evaluation functor $ev \colon \C \to \Set^\Sp$ is coherent.
\end{proposition}

In particular, the evaluation functor induces a geometric morphism $\phi_{ev} \colon \Set^\Sp \to Sh(\C, J_{coh})$. Concretely, this geometric morphism is given by the adjoint pair $L_{ev} \colon Sh(\C, J_{coh})\leftrightarrows \Set^\Sp \colon R_{ev}$. Here $L_{ev}$ is the free colimit extension of $ev$ restricted to $Sh(\C, J_{coh})$:
$$
\xymatrix{
\C \ar[r]^-{y} \ar[rd]_-{ev} & \Set^{\C^{\op}} \ar@{.>}[d] \ar@<0.2pc>[r]^-{a} & Sh(\C, J_{coh}) \ar[ld]^-{L_{ev}} \ar@<0.2pc>[l]^-{i}\\
& \Set^\Sp
}
$$
and $R_{ev}$ may be described as
$$
\begin{array}{rcl}
R_{ev} \colon \Set^\Sp &\to& Sh(\C, J_{coh})\\
H &\mapsto& Hom(ev(\_), H).
\end{array}
$$

In the remainder the notion of \emph{type} will play an important role. Remark that we use the word `type' as in model theory (and not as in type theory).

\begin{definition}
For a model $M$, $A \in \C$ and $a \in M(A)$ we define
$$
t_A(a,M) = \{U \in Sub_\C(A)\,|\, a \in M(U)\}
$$
and we call $t_A(a,M)$ the \emph{type of $a$ in $M(A)$}.
\end{definition}

From now on we restrict our attention to full subcategories $\Sp$ of $Mod(\C)$ satisfying the following conditions (see also \cite{Mak}).
\begin{itemize}
\item[(M1)] for all $M \in \Sp$, $M$ is a p-model (of $\C$ in $\Set$, as in Definition~\ref{def:pmod});
\item[(M2)] for all $A \in \C$, $\rho$ prime filter in $Sub_\C(A)$, there exist $M \in \Sp$ and $a \in M(A)$ s.t. $\rho = t_A(a, M)$;
\item[(M3)] for all $A \in \C$, $M, N \in \Sp$, $a \in M(A)$, $b \in N(A)$, if $b \in \bigwedge\{N(U)\,|\, U \in t_A(a, M)\}$, then there exists a morphism $h \colon M \to N$ in $\Sp$ s.t. $b =h_A(a)$.
\end{itemize}

Let us first remark that such classes of models indeed exist. For example, the class of $\lambda$-special models of $\C$ satisfies the above requirements.\footnote{Here $\lambda$ is a cardinal depending on the size of $\C$.}
The notion of special model is a generalisation of the notion of saturated model, in the sense that any saturated model is special. Only under the assumption of the Generalised Continuum Hypothesis, every consistent theory has a saturated model. This assumption is not required for the existence of special models. For more background on special models the reader is refered to 
\cite{CK, Hodg}.

For a coherent category $\C$, any class of models $\Sp$ which satisfies the above requirements contains enough models to faithfully represent $\C$. That is, in this case, the evaluation functor $\C \to \Set^\Sp$ is conservative and therefore the induced geometric morphism 
$$
\phi_{ev} \colon \Set^\Sp \to Sh(\C,J_{coh})
$$ 
is a surjection. We will show in Theorem~\ref{thm:hcfev} that the geometric morphism $\phi_t \colon T(\C) \to Sh(\C,J_{coh})$ (see~\eqref{eq:phit}) is the localic part of $\phi_{ev}$. Before we embark on the proof, we first make a short remark about the logical intuition behind this statement. A geometric morphism $\mathcal{F} \to \mathcal{E}$ may be viewed as an expansion of the theory of $\mathcal{E}$ with sorts, function symbols, relation symbols and axioms. In case the geometric morphism is localic, no new sorts are added to the theory. That is, a localic geometric morphism corresponds to an expansion of the theory with only function symbols, relation symbols and axioms. When forming the hyperconnected localic factorisation of a geometric morphism, one splits up the corresponding expansion of the theory in two steps: first one only adds the function and relation symbols which concern sorts of the original theory and are definable in the new theory, and as axioms one adds the statements which are expressible in the language of the original theory and are derivable in the new theory; then one completes the expansion of the theory in the second step. From this point of view, Theorem~\ref{thm:hcfev} intuitively states that, for a coherent category $\C$, the topos of types of $\C$ contains the information about the theory corresponding  to $\C$ which may be derived when studying the models of $\C$. To prove Theorem~\ref{thm:hcfev}, we start with a lemma.

\begin{lemma}
\label{prop:evpmod}
Let $\Sp$ be a full subcategory of $Mod(\C)$ satisfying (M1)-(M3). The evaluation functor $ev \colon \C \to \Set^\Sp$ is a p-model.
\end{lemma}

\begin{proof}
Let $A \xrightarrow{\alpha} B$ in $\C$ and $\rho$ a prime filter in $Sub_\C(A)$. We have to show that the following two subobjects of the functor $\exists_{ev(\alpha)}(ev(A))\colon \Sp \to \Set$,
$$
\begin{array}{rcl}
G_1 &:=& \exists_{ev(\alpha)}(\bigwedge\{ev(U)\,|\,U \in \rho\}) \\
G_2 &:=& \bigwedge\{\exists_{ev(\alpha)}(ev(U))\,|\, U \in \rho\},
\end{array}
$$
are equal. We compute $G_1$ and $G_2$ at $M \in \Sp$. We use the fact that in a presheaf category $\Set^{\cat{D}^{\op}}$, for a natural transformation $G \xrightarrow{\tau} H$, the action of the left adjoint $\exists_\tau \colon Sub(G) \to Sub(H)$ is given by taking component-wise direct images, {\it i.e.}, for $G' \in Sub(G)$, $D \in \D$,
$
\exists_\tau(G')(D) = \tau_D[G'(D)].
$

$$
\begin{array}{rcll}
G_1(M) 
&\cong&
\exists_{ev(\alpha)}(\bigwedge\{ev(U)\,|\,U \in \rho\})(M)\\
&\cong&
ev(\alpha)_M[(\bigwedge\{ev(U)\,|\,U \in \rho\})(M)]&\\
&\cong&
M(\alpha)[\bigwedge\{ev(U)(M)\,|\,U \in \rho\}]\\
&\cong&
M(\alpha)[\bigwedge\{M(U)\,|\,U \in \rho\}]
\end{array}
$$
and
$$
\begin{array}{rcll}
G_2(M) 
&\cong&
(\bigwedge\{\exists_{ev(\alpha)}(ev(U))\,|\, U \in \rho\})(M)\\
&\cong&
\bigwedge\{(\exists_{ev(\alpha)}(ev(U)))(M)\,|\,U \in \rho\}\\
&\cong&
\bigwedge\{ev(\alpha)_M[(ev(U))(M)]\,|\,U \in \rho\}\\
&\cong&
\bigwedge\{M(\alpha)[M(U)]\,|\,U \in \rho\}.
\end{array}
$$
As, by (M1), for all $M \in \Sp$, $M$ is a p-model, $G_1(M) \cong G_2(M)$. \QED
\end{proof}

\begin{theorem}
\label{thm:hcfev}
Let $\Sp$ be a full subcategory of $Mod(\C)$ satisfying (M1)-(M3). The hyperconnected localic factorisation of $\phi_{ev} \colon \Set^\Sp \to Sh(\C, J_{coh})$ is given by
$$
\xymatrix{
&T(\C) \ar[rd]^-{\phi_t}\\
\Set^\Sp \ar[rr]_-{\phi_{ev}} \ar[ru] && Sh(\C, J_{coh})
}
$$
\end{theorem}

\begin{proof}
Let $\Omega_{\Set^\Sp}$ be the subobject classifier of $\Set^\Sp$. The hyperconnected localic factorisation of $\phi_{ev}$ is given by
$$
\xymatrix{
&Sh_{Sh(\C, J_{coh})}(R_{ev}(\Omega_{\Set^\Sp})) \ar[rd]\\
\Set^\Sp \ar[rr]_-{\phi_{ev}} \ar[ru] && Sh(\C, J_{coh})
}
$$
As the topos of types $T(\C)$ is the topos of sheaves over the internal locale $\S_\C^\delta$ in $Sh(\C, J_{coh})$, it suffices to show $\S_\C^\delta \cong R_{ev}(\Omega_{\Set^\Sp})$ (as locales in $Sh(\C, J_{coh}$)). Recall that, for $A \in \C$,
$$
\begin{array}{rcl}
R_{ev}(\Omega_{\Set^\Sp})(A) &=&
Hom_{\Set^\Sp}(ev(A), \Omega_{\Set^\Sp})\\
&=&
Sub_{\Set^\Sp}(ev(A)).
\end{array}
$$
Let $\sigma \colon \S_\C \to R_{ev}(\Omega_{\Set^\Sp})$ be the natural transformation given by, for $A \in \C$,
$$
\begin{array}{rcll}
\sigma_A \colon \S_\C(A) &\to& Sub(ev(A)) = R_{ev}(\Omega_{\Set^\Sp})(A)\\
U &\mapsto& ev(U).
\end{array}
$$
Note that naturality of $\sigma$ follows from the fact that he evaluation functor $ev \colon \C \to \Set^\Sp$ preserves finite limits, whence, for $A \xrightarrow{\alpha} B$ in $\C$ and $U \in Sub_\C(A)$,
$$
ev(\alpha^*(U)) = ev(\alpha)^*(ev(U)) \,\,\,\text{in $Sub(ev(A))$}.
$$
As $\sigma_A$ is a lattice homomorphism and $R_{ev}(\Omega_{\Set^\Sp})(A) \in \DL^+$, this map extends uniquely to a complete lattice homomorphism $\overline{\sigma_A} \colon \S_\C^\delta (A) \to R_{ev}(\Omega_{\Set^\Sp})(A)$, given by the natural isomorphism \eqref{eq:natdl}. We will prove the theorem by showing that the components $\overline{\sigma_A}$ constitute an internal frame isomorphism $\overline{\sigma} \colon \S_\C^\delta \to R_{ev}(\Omega_{\Set^\Sp})$. 

\auxproof{
Proof of the fact that $Sub_{\Set^\Sp}(ev(A)) \in DL^+$: As meets and joins in $Sub_{\Set^\Sp}(ev(A))$ are computed component-wise, it is immediate that this is a completely distributive lattice. To prove that it is join-generated by its completely join-irreducible elements, let $H \hookrightarrow ev(A)$. Define, for $N \in \Sp$ and $a \in H(N)$, $H_{N,a}$ to be the smallest subsheaf of $ev(A)$ such that $a \in H_{N,a}(N)$. Then $H_{N,a}$ is completely join-irreducible in $Sub_{\Set^\Sp}(ev(A))$ and $H = \bigvee \{H_{N,a}\,|\,N \in \Sp, a \in H(N)\}$.
}

\begin{enumerate}
\renewcommand{\theenumi}{\roman{enumi}}
\item $\overline{\sigma}$ is a natural transformation.\\
Let $A \xrightarrow{\alpha} B$ in $\C$. We have to show that the following diagram commutes.
$$
\xymatrix@C=3pc{
S_\C^\delta(B) \ar[r]^-{(\alpha^*)^\delta} \ar[d]_-{\overline{\sigma_B}} & S^\delta(A) \ar[d]^-{\overline{\sigma_A}}\\
Sub(ev(B)) \ar[r]_-{ev(\alpha)^*} & Sub(ev(A))
}
$$
As $ev(\alpha)^*$ is a complete lattice homomorphism, this follows from the naturality of $\overline{(\_)}$ and the naturality of $\sigma$:
$$
\overline{\sigma_A} \after (\alpha^*)^\delta \,=\, \overline{\sigma_A \after \alpha^*} \,=\, \overline{ev(\alpha)^* \after \sigma_B} \,=\, ev(\alpha)^* \after \overline{\sigma_B}.
$$

\item $\overline{\sigma}$ is an internal frame homomorphism.\\
Let $A \in \C$. As $\overline{\sigma_A}$ is a morphism in $\DL^+$, {\it i.e.}, a compete lattice homomorphism, it is in particular a frame homomorphism. It is left to show that $\overline{\sigma}$ preserves existential quantification. Let $A \xrightarrow{\alpha} B$ be a morphism in $\C$ and consider
$$
\xymatrix{
S_\C^\delta(A) \ar[r]^-{(\exists_\alpha)^\delta} \ar[d]_-{\overline{\sigma_A}} & S^\delta(B) \ar[d]^-{\overline{\sigma_B}}\\
Sub(ev(A)) \ar[r]_-{\exists_{ev(\alpha)}} & Sub(ev(B))
}
$$
We use Proposition~\ref{prop:comjpm} to prove commutativity of the above diagram.
Recall that, for $U \in Sub_\C(A)$, $\sigma_A(U) = ev(U)$. From the fact that $ev \colon \C \to \Set^\Sp$ is a coherent functor it follows that, for $U \in Sub_\C(A)$, $ev(\exists_\alpha U) = \exists_{ev(\alpha)}(U)$, {\it i.e.}, $\sigma_B \after \exists_{\alpha} = \exists_{ev(\alpha)} \after \sigma_A$.
As, by Lemma~\ref{prop:evpmod}, $ev$ is a p-model, we have, for every prime filter $\rho$ in $Sub_\C(A)$,
$$
\exists_{ev(\alpha)}(\bigwedge\{ev(U)\,|\,U \in \rho\}) = \bigwedge \{\exists_{ev(\alpha)}(ev(U))\,|\,U \in \rho\}.
$$
Hence, Proposition~\ref{prop:comjpm} applies, from which it follows that the diagram commutes.

\item For all $A \in \C$, $\overline{\sigma_A}$ is an embedding.\\
Let $u, v \in \S_\C^\delta(A)$ with $\overline{\sigma_A}(u) \le \overline{\sigma_A}(v)$. We have to show $u \le v$. As $\S_\C^\delta(A)$ is join generated by its completely join-irreducible elements, it suffices to show that, for all $x \in J^\infty(S_\C^\delta(A))$, $x \le u$ implies $x \le v$. Let $x \in J^\infty(S_\C^\delta(A))$ with $x \le u$. Consider
$$
\begin{array}{rcl}
\rho_x &=& \{U \in Sub_\C(A)\,|\,x \le U\}.
\end{array}
$$
By property (M2) of $\Sp$ there exist $M \in \Sp$ and $a \in M(A)$ s.t. $\rho_x = \{U \in Sub_\C(A)\,|\,a \in M(U)\}$. Note that
$$
\begin{array}{rcl}
\overline{\sigma_A}(u)(M) 
&=&
\bigvee\{\overline{\sigma_A}(z)\,|\,u \ge z \in J^\infty(S_\C^\delta(A))\}(M)\\
&=&
\bigvee\{\bigwedge\{ev(U)(M)\,|\,U \in \rho_z\}\,|\,u \ge z \in  J^\infty(S_\C^\delta(A))\}\\
&=&
\bigvee\{\bigwedge\{M(U)\,|\,U \in \rho_z\}\,|\,u \ge z \in  J^\infty(S_\C^\delta(A))\}.
\end{array}
$$
In particular $u \ge x$ and $\bigwedge\{M(U)\,|\,U \in \rho_x\} = \bigwedge\{M(U)\,|\,a \in M(U)\}$. Hence $a \in \overline{\sigma_A}(u)(M)$.

As, by assumption, $\overline{\sigma_A}(u) \le \overline{\sigma_A}(v)$, $a \in \overline{\sigma_A}(v)(M)$. Hence, there exists $z \in J^\infty(S_\C^\delta(A))$ s.t. $v \ge z$ and $a \in \bigwedge\{M(U)\,|\,U \in \rho_z\}$. It follows that
$
\rho_z \subseteq \{U \in Sub_C(A)\,|\,a \in M(U)\} = \rho_x.
$ 
Hence $x \le z \le v$, as required.

\item For all $A \in \C$, $\overline{\sigma_A}$ is surjective.\\
Recall that $\overline{\sigma_A} \colon \S_\C^\delta(A) \to Sub(ev(A))$. Let $H \hookrightarrow ev(A)$. We define, for $N \in \Sp$ and $a \in H(N) \subseteq ev(A)(N) = N(A)$,
$$
\begin{array}{rcl}
\rho_{N, a} &=& \{U \in Sub_\C(A)\,|\, a \in N(U)\}.\\
\end{array}
$$
Note that $\rho_{N,a}$ is a prime filter in $Sub_\C(A)$ and let $x_{N,a}$ be the corresponding completely join-irreducible in $\S_\C^{\delta}(A)$, {\it i.e.}, $x_{N,a} = \bigwedge \rho_{N,a}$. We set
$$
\begin{array}{rcl}
u &=& \bigvee \{x_{N,a}\,|\, N \in \Sp, a \in H(N)\}
\end{array}
$$
and we will show $\overline{\sigma_A}(u) = H$ in $Sub(ev(A))$. Let $M \in \Sp$. We have to show $H(M) \cong \overline{\sigma_A}(u)(M)$. By definition
$$
\begin{array}{rcl}
 \overline{\sigma_A}(u)(M)
&\cong&
\overline{\sigma_A}(\bigvee \{x_{N,a}\,|\, N \in \Sp, a \in H(N)\})(M)\\
&\cong&
\bigvee\{\bigwedge\{\overline{\sigma_A}(U)\,|\,U \in \rho_{N,a}\}\,|\,N \in \Sp, a \in H(N)\}(M)\\
&\cong&
\bigvee\{\bigwedge\{ev(U)(M)\,|\,U \in \rho_{N,a}\}\,|\,N \in \Sp, a \in H(N)\}\\
&\cong&
\bigvee\{\bigwedge\{M(U)\,|\,U \in \rho_{N,a}\}\,|\,N \in \Sp, a \in H(N)\}.
\end{array}
$$
It is easy to see that $H(M) \subseteq \overline{\sigma_A}(u)(M)$. For suppose $b \in H(M)$. Then $b \in \bigwedge\{M(U)\,|\,U \in \rho_{M,b}\}$, whence $b \in \overline{\sigma_A}(u)(M)$. For the converse inclusion, suppose $b \in \overline{\sigma_A}(u)(M)$. Then there exist $N \in \Sp$ and $a \in H(N)$ s.t. $b \in \bigwedge\{M(U)\,|\,U \in \rho_{N,a}\}$. By property (M3) of $\Sp$ there exists $h \colon M \to N$ s.t. $b = h_A(a)$. As $a \in H(N)$ and $H$ is a subfunctor of $ev$ it follows that $b = h_A(a) \in H(M)$:
$$
\xymatrix{
H(N) \ar[r]^-{h_A} \ar[d] & H(M) \ar[d]\\
ev(A)(N)  \ar[r]_{h_A} & ev(A)(M)
}
$$
\end{enumerate}
This completes the proof that $\overline{\sigma}$ is an internal frame isomorphism.\QED
\auxproof{
A natural transformation is an isomorphism if and only if all its components are isomorphisms. Hence, $\sigma$, viewed just as a natural transformation, has an inverse $\tau$, which may be obtained by taking component-wise inverses. As all components of $\sigma$ are frame isomorphisms, also the components of $\tau$ are frame homomorphisms. To show that $\tau$ is an internal frame morphism it is left to show that it preserves existential quantification. Let $A \xrightarrow{\alpha} B$ in $\C$ and consider
$$
\xymatrix{
Sub(ev(A)) \ar[r]^-{\exists_{ev(\alpha)}} \ar[d]_-{\tau_A} & Sub(ev(B)) \ar[d]_-{\tau_B}\\
S_\C^\delta(A) \ar@/_/[u]_-{\sigma_A} \ar[r]_-{(\exists_{\alpha})^\delta} & \S_\C^\delta(B) \ar@/_/[u]_-{\sigma_B}
}
$$
Note that
$$
\begin{array}{rcll}
\sigma_B \after (\exists_\alpha)^\sigma \after \tau_A
&=&
\exists_{ev(\alpha)} \after \sigma_A \after \tau_A &\text{($\sigma$ preserves $\exists$)}\\
&=&
\exists_{ev(\alpha)}\\ 
&=&
\sigma_B \after \tau_B \after \exists_{ev(\alpha)}. 
\end{array}
$$
As $\sigma_B$ is an isomorphism, it follows that $ (\exists_\alpha)^\sigma \after \tau_A =  \tau_B \after \exists_{ev(\alpha)}$.
} 
\end{proof}

\begin{proof}[Aknowledgements] 
I would like to thank Mai Gehrke for suggesting this topic to me and for her stimulating guidance along the way. Furthermore, I would like to thank Ieke Moerdijk and Sam van Gool for the interesting and helpful discussions and for their valuable comments on earlier versions of this paper. I also thank Michael Makkai and Gonzalo Reyes for hosting me for an inspiring week in Montreal. Finally, I would like to thank the referee for his suggestions.
\end{proof}

\bibliographystyle{amsplain}
\bibliography{thesisDC}

\end{document}